\documentclass[11pt]{article}

\usepackage{fullpage}

\usepackage{amsmath,amssymb,amsfonts, amsthm}    % For better support of math
\usepackage{graphicx}
\usepackage[english]{babel}

\usepackage{bbm, dsfont}

\usepackage{enumerate}

\usepackage{times}
\usepackage[T1]{fontenc}
\usepackage{multimedia}
\usepackage{color}
\usepackage{verbatim}

\newcommand{\RR}{\mathbb{R}}
\newtheorem{theorem}{Theorem}
\newtheorem{lemma}{Lemma}

\numberwithin{equation}{section}

\theoremstyle{definition}
\newtheorem{definition}{Definition}

\def\RR{\mathbb{R}}

\title{Prediction with Expert Advice: a PDE
Perspective\footnote{This research was partially supported
by NSF grant DMS-1311833.}
}
\author{
Nadejda Drenska\footnote{Department of Mathematics,
University of Minnesota; ndrenska@umn.edu. This work is a refinement
of the first author's PhD thesis, {\it A PDE Approach to a Prediction
Problem Involving Randomized Strategies}, NYU, 2017.} \ and
Robert V. Kohn\footnote{Courant Institute of Mathematical Sciences, New
York University; kohn@cims.nyu.edu}
}
\date{\today}     % change \today to a particular date

\begin{document}

\maketitle

\bigskip
\noindent

This work addresses a classic problem of online prediction with expert
advice. We assume an adversarial opponent, and we consider both the
finite-horizon and random-stopping versions of this zero-sum, two-person game.
Focusing on an appropriate continuum limit and using methods from optimal
control, we characterize the value of the game as the
viscosity solution of a certain nonlinear partial differential equation.
The analysis also reveals the predictor's and the opponent's minimax
optimal strategies. Our work provides, in particular, a continuum perspective
on recent work of Gravin, Peres, and Sivan (Proc SODA 2016). Our techniques
are similar to those of Kohn and Serfaty (Comm Pure Appl Math 2010), where
scaling limits of some two-person games led to elliptic or parabolic PDEs.

%%%%%%%%%%%%%%%%%%%%%%%%%%%%%%%%%%%%%%%%%%%%%%%%%%%%%%%%%%%%

\section{Introduction}

Our work addresses a problem involving `prediction with expert
advice.' This is a well-established framework in which a player
tries to use `expert advice' to invest optimally (for the worst case
scenario) against an adversarial market. The measure of effectiveness
of the player's strategy is regret minimisation: performance under the
metric of `regret', or distance between a player's performance and that of
the (retrospectively) best-performing `expert'. We use linear
regret, in other words the difference between a player's loss and
an expert's loss. Here, `prediction' is not about modelling a time series
probabilistically; instead, the player tries to synthesise the
advice of the experts in a way that guarantees good performance
in a worst case setting.

We consider the following setup. There are two entities -- a 'player' and
a 'market' -- and a fixed number $n$ of 'experts'. The market chooses
which experts win or lose at every time step. The player chooses
which expert to listen to at each time step. The two entities'
optimal strategies are mixed, i.e. the strategies involve probability
distributions over the space of available outcomes. The player's goal
is to accumulate overall winnings as close as possible to those of
the best performing expert at the 'end' of the game (assuming that
the market works against the player). There are two variants: one
with a fixed stopping time ('the finite horizon problem') and one
where the stopping time is random with a constant probability of
stopping at every time step ('the geometric stopping problem').
The goal in each variant is to identify the optimal strategies
of the player and the market, as well as the associated value function.

The general approach is `numerical analysis in reverse' -- interpreting
each discrete formulation as a numerical scheme for an appropriate
nonlinear PDE. We prove that the solution to the discrete problem
is asymptotically close to the unique viscosity solution of the PDE;
as a result, knowledge of the PDE solution provides an indication
about the optimal strategy for the discrete game. The 'finite horizon
problem' leads to a parabolic PDE, whereas the 'geometric stopping
problem' is associated to an elliptic PDE.

The overall outline of our analysis is as follows.
Firstly, for each variant we define a discrete approximation scheme
associated with a dynamic programming principle for the game. For
the geometric stopping problem the existence of a solution to the
scheme is nontrivial. Its construction relies on a time dependent
problem which is run to equilibrium (or equivalently, a contraction
mapping argument). For the finite horizon problem, existence of
a solution to the scheme is easily established by induction.
Convergence of each scheme is obtained through standard viscosity
technology: the scheme is stable, monotone, and consistent, hence
its solution converges to the unique viscosity solution of the
PDE. (Our proof uses the framework of Barles and Souganidis \cite{BS},
adjusted to accommodate the special features of our problem.)
Finally, we give an explicit solution for the elliptic PDE
associated to the geometric stopping problem with three experts
(it is the continuous analogue of the solution obtained using
discrete methods in Gravin, Peres, and Sivan's paper \cite{GPS}).

Our work shows that although online machine learning is not in
any conventional sense a stochastic control problem, continuous
methods are useful for its analysis (in much the same way that
PDEs are useful for studying stochastic control). It should be
noted that we are not the first to apply PDE methods
to an online machine learning problem. Indeed, Kangping Zhu's
thesis \cite{Zhu} used PDE methods to achieve a similar
goal in a somewhat different setting.

To put this work in context, we briefly review some of the
machine learning literature on prediction with expert advice. Most of this
work focuses on regret bounds (e.g. using specific strategies to prove
upper bounds on the predictor's regret). A prediction
problem appears in Cover's article \cite{Cover} as far back as 1965,
where he establishes an $O(\sqrt{T})$ regret bound, where $T$ is the
number of rounds played; Cover also solves the problem for $n=2$.
A classical treatment is available in Cesa-Bianchi and Lugosi's
book \cite{CBL}; it outlines the theoretical foundation of the
area and provides a self-contained treatment of many results,
including an upper bound on the regret of order $O(\sqrt{T\log n})$,
proved using a well-chosen multiplicative weight algorithm.
Some earlier, foundational works include Vovk's \cite{Vovk} and
Littlestone and Warmuth's \cite{LW}; they introduced the weighted
majority algorithm as a method the predictor can use to weight the
experts' bids. Haussler et all \cite{HKW} achieve a
$\Omega(\sqrt{T\log n})$ regret bound in the case of absolute loss.
Abernethy et al \cite{A} consider a game played until a fixed
number of losses is incurred by an expert. Luo and Schapire \cite{LS}
investigate a version of the game with a randomly chosen final time.
In \cite{RSS} Rakhlin et al. present algorithms using ''random play
out''. A recent paper by Gravin, Peres, and Sivan \cite{GPS} analyzes
the same problems that we consider here. That work uses discrete methods
and connections to random walks; ours can be viewed as providing
its continuous-time analogue. For more detail on the relationship
between our work and \cite{GPS}, see Subsection \ref{GPScomp}.
Our PDE characterization of the value function has already seen an
interesting application: in \cite{BEZ}, Bayraktar et al use it
to obtain an explicit solution for the geometric stopping version
of the game with $n=4$ experts.

There are other instances in the literature where scaling
limits of multistep decision processes lead to parabolic
or elliptic PDEs. For example, the work of Kohn
and Serfaty on two-person game interpretations of motion
by curvature \cite{KS1} and many other PDE problems \cite{KS2} has this
character. So does the work of Peres, Sheffield, Schramm, and
Wilson connecting the `tug-of-war' game to the infinity-Laplacian
\cite{PS2} and the p-Laplacian \cite{PS1} (this work has
seen many extensions, e.g. \cite{APSS}, \cite{AS1}, \cite{LM}, \cite{NS}).

A particular advantage of our treatment is that it is not limited
to the classical payoff function in the online machine learning
literature, namely regret with respect to the best expert
$\varphi(x)=\max_k \{x_k\}$, where $x_k$ is regret with
respect to expert $k$. In fact, it works for a more general
class of payoff functions, namely functions $\varphi$ that
are globally Lipschitz continuous, non-decreasing, symmetric
in their dependent variables $x_k$, have linear growth at $\infty$,
and satisfy $\varphi (x_1+c, ... , x_n+c ) = \varphi(x_1,...,x_n) + c$.
Different choices of $\varphi$ represent generalizations of the
classic linear regret performance measure. We prove results for the
general class of payoff functions described above; we
restrict $\varphi$ to $\varphi(x)=\max_k \{x_k\}$ only
to find the explicit solution of the $n=3$ elliptic case.

The outline of this paper is as follows. In section \ref{NF} we
introduce notation and the discrete formulation of the problem we
wish to solve, as well as the dynamic programming principle (DPP)
for each case. In section \ref{DPP} we derive heuristically the
associated PDEs. In section \ref{GS} we prove that both in the
finite horizon and in the geometric stopping cases the discrete
dynamic programming principle introduced in section \ref{DPP}
has a unique at most linear growth solution. In section \ref{EU}
we cite results showing that each of our PDEs has a unique
solution among functions with at most linear growth. In
section \ref{CV} we relate the discrete solutions to the
solutions of the PDEs by proving that the solutions of the
appropriately scaled DPP solve the appropriate PDE in the
limit $\varepsilon \to 0$.  In section \ref{ES} we investigate
the particular case of $n = 3$ experts in the geometric
stopping problem, and provide an explicit formula for the
solution of the PDE.

%%%%%%%%%%%%%%%%%%%%%%%%%%%%%%%%%%%%%%%%%%%%%%%%%%%%%%%%%%%%%%%%

\section{Notation and Formulation\label{chap:one}}\label{NF}

In this section we introduce our notation and formulate the two variants of our problem. We start in \ref{Not} with the basic setup; subsections \ref{fhp} and \ref{gsp} present the two classical variants of the game (described in detail, for example, in \cite{GPS}). Lastly, in \ref{sg} we present the scaled variants of the game.

\subsection{Notation}\label{Not}

We will be considering a game with randomized strategies but let us focus on a non-probabilistic set up first. There are two entities - a 'market' and a 'player' -- as well as $n$ experts denoted by $1, 2,..., n$. The game is played for $T$ rounds (in the 'finite horizon' problem), or else with a random stopping time (using a fixed probability $\delta$ of stopping at each time step -- we call this the 'geometric stopping' problem). At each round $t$, every expert $k$ makes a prediction (say, whether stock $k$ will go up or down), and the player chooses to follow a particular expert, say the $l$th one. The market determines the gains $v_{t,k}$ of each expert $k$ ($v_{t,k}=1$ if expert $k$ made an accurate prediction at round $t$ and $v_{t,k}=0$ otherwise). Then the outcomes of the player and the market are revealed. We denote by $x_k$ the player's 'regret with respect to expert $k$'; this is, by definition, expert $k$'s cumulative gains minus the player's cumulative gains. Thus the increment of $x_k$ at time $t$ is
$$\Delta x_{t,k}=v_{t,k} -v_{t,l}$$
if the player follows expert $l$. 

The game we study is similar to the one just described, except that the player and the market choose
{\it randomized} strategies:
\begin{itemize}
\item At each step $t$, without knowing the player's move, the market chooses a probability distribution $p_t$, over all the possible outcomes for the $n$ experts, which we represent by vectors $\vec v \in \{0,1\}^n$. (An
    outcome is thus a choice of the subset of experts making correct predictions; for example, if all the experts are correct then $v= \vec 1$ is the vector of ones.)

\item Simultaneously, at every turn $t$ without knowing the market's move, the player chooses a probability distribution over the $n$ experts, i.e. a vector $\vec \alpha_t = ( \alpha_{t,1}, \alpha_{t,2}, ... , \alpha_{t,n})$, where $\sum \alpha_{t,i} =1$ and $\alpha_{t,i} \geq 0$. Its meaning is that the player follows expert $l$ at time $t$ with probability $\alpha_{t,l}$ (obtaining the same outcome as expert $l$, namely $v_{t,l}$).

\item The player seeks to maximize (and the market seems to minimize) the expected final-time regret (the
expectation being taken with respect to probabilities associated with the randomized strategies).
\end{itemize}

The state variables $x_j$ for this game are the player's regret with respect to the $j^{th}$ expert, meaning expert's gain minus player's gain. At risk of redundancy, we emphasize that market and the player know they are playing against each other, and this influences their optimal strategies. The player chooses the probability distribution on $\vec \alpha_t$ so as to minimize her expected regret at the end of the game; meanwhile the market chooses the probability distribution $ p_t$ which maximizes expected regret at the end of the game. These distributions are not fixed throughout the game and will depend on various unknowns, and on which version of the game is being considered (the 'finite horizon' version or the 'geometric stopping' one).

For notational convenience, whenever we look at the player's optimization subject to $\sum \alpha_i =1$ and $\alpha_i \geq 0$, we will write this choice as  $\min_{player}$. Similarly, whenever the market chooses an optimal probability distribution $p$ on the set of all possible choices $v \in \{0,1\}^n$, we denote the market's maximization with $\max_{market}$. We write $\mathbf{E}$ for the expected value over the mixed strategies. Lastly, whenever the market chooses a probability distribution  $p$ on the set of all possible choices $v \in \{0,1\}^n$, subject to the condition of 'balance', i.e.
 \begin{equation} \label{bal}
\mathbf{E}_p [v_i] = \mathbf{E}_p [v_k]  \quad \forall i,k \quad \fbox{balance condition},
\end{equation}
we denote this by $\max_{balance}$.

As the final time measure of regret, we consider an arbitrary function $\varphi(x_1, ..., x_n)$ that satisfies the following properties:
\begin{align}
 & \text{$\varphi$ is globally Lipschitz continuous},\\
  &\text{non-decreasing in each variable},\\
 & \text{symmetric in its dependent variables}~ x_1, ..., x_n,\\
 & |\varphi(x)|\leq C_1|x|+C_2~~~ \text{(a consequence of (2.2)), and}\\
& \text{for every $c$ it holds that}~~ \varphi (x_1+c, ... , x_n+c ) = \varphi(x_1,...,x_n) + c.
\end{align}
One such function $\varphi$ is $\varphi(x) = \max_k x_k$.

\subsection{The Finite Horizon Problem}\label{fhp}

The finite horizon problem is to determine the player's expected regret (the value function of the game) and the associated optimal strategies for both the player and the market, provided that the game ends at an a priori fixed time $T$ and starts at time $t$ such that $t \in \mathbb{N}, t \leq T$ with initial regret vector $x$. One can write the value function through a dynamic programming principle (DPP): it is the expected payoff at final time, provided the player and the market play optimally against each other, in particular doing the best that could be done after one time step. Through the dynamic programming principle, the discrete finite horizon formulation becomes:

\begin{align}\label{TDPP}
 & w^d(t,x) = \min_{player} \max_{market} \mathbf{E}[w^d(t+1, x+ \Delta x)]  \\
 & w^d(T,x) = \varphi(x). \nonumber
\end{align}

\subsection{The Geometric Stopping Problem}\label{gsp}

The geometric stopping problem is to determine the player's expected regret (the value function) provided the game starts at regret vector $x$. The game either stops with probability $\delta$, $0<\delta<1$, in which case the payoff is $\varphi(x)$; or else it continues, with probability $1-\delta$, for at least one more round, with player and market playing against each other optimally. One can thus express the value function through a DPP:
\begin{equation} \label{GDPP}
u^d(x) = \delta \varphi(x) +(1-\delta)\min_{player} \max_{market} \mathbf{E}[u^d(x+\Delta x)].
\end{equation}
Observe that there is no time-dependence in this case. (The probability of stopping, $\delta$, is assumed constant, i.e. independent of time).

\subsection{The Scaled Games}\label{sg}

Since we are interested in the behavior of the games over long periods of time, we consider scaled versions 
of them. For the finite horizon problem we scale spatial steps to be $0$ and $\varepsilon$ (instead of $0$ and $1$) and time steps to be $\varepsilon^2$ (instead of $1$), so the game is played for $T/\varepsilon^2$ steps. The reason for this scaling is that we expect to obtain a parabolic PDE in the limit. Then, the analogue of  equation  (\ref{TDPP})  is:
\begin{align}\label{sTDPP}
 & w(t,x) = \min_{player} \max_{market} \mathbf{E}[w(t+\varepsilon^2, x+\varepsilon \Delta x)]  \\
 & w(T,x) = \varphi(x). \nonumber
\end{align}

For the geometric stopping case (\ref{GDPP}) we observe that the expected number of rounds until stopping is $1/ \delta$ , since the probability of stopping after any step is $\delta$. We choose, just as in the previous case, to have spatial steps $\varepsilon$, and a typical number of steps of order $\varepsilon^{-2}$ , hence we choose $\delta =\varepsilon^2$.
 The analogue of (\ref{GDPP}) is thus:
\begin{equation}\label{sGDPP}
 u(x) = \varepsilon^2 \varphi(x) +(1-\varepsilon^2)\min_{player} \max_{market} \mathbf{E}[u(x+\varepsilon \Delta x)].
\end{equation}

The goal of this work is to investigate the limiting behavior of the solutions of (\ref{sTDPP}) and (\ref{sGDPP}). A key observation is that the statements of the DPP, as $\varepsilon \to 0$ are semi-discrete numerical schemes for  corresponding PDEs. We prove that the solution of  (\ref{sTDPP}) converges  to that of the parabolic problem
\begin{align}
 & w_t(t,x) + \frac{1}{2}\max_{v\in\{0,1\}^n}\langle D^2w(t,x)\cdot v, v\rangle = 0, \label{parE} \\
 & w(T,x) = \varphi(x), \nonumber
\end{align}
as $\varepsilon$ goes to $0$, whereas the solution of (\ref{sGDPP}) converges to that of
\begin{equation}
u(x) = \varphi(x) +  \frac{1}{2}\max_{v\in\{0,1\}^n}\langle D^2u(x)\cdot v, v\rangle  \label{Equa}
\end{equation}
as $\varepsilon$ goes to $0.$

A central question is whether the scaled games are equivalent to the unscaled ones. Whenever $\varphi$ satisfies $\varepsilon \varphi(\frac{x}{\varepsilon}) = \varphi(x)$, the answer is yes. In particular it is true for the classical choice of regret  $\varphi(x)=\max_k \{x_k\}$. For the finite horizon case, let the discrete-in-time, continuous-in-space function $w^d$ solve (\ref{TDPP}) and define
$$ \tilde{w}(\tau, y) =\varepsilon w^d(\frac{\tau}{\varepsilon^2} ,\frac{y}{\varepsilon}).$$
It satisfies
$$\tilde{ w}(t,x) = \min_{player} \max_{market} \mathbf{E}[\tilde{w}(t+\varepsilon^2, x+\varepsilon \Delta x)]$$
with $\tilde{w}(T, x)=\varepsilon\varphi(\frac{x}{\varepsilon})$ at the final time. So if $\varepsilon\varphi(\frac{x}{\varepsilon})=\varphi(x)$, $\tilde{w}$ is the solution of  (\ref{sTDPP}).
The situation with the geometric stopping case is similar. We scale  $$\tilde{u}( y) =\varepsilon u^d(\frac{y}{\varepsilon})$$
and take $\delta = \varepsilon^2$ in (\ref{GDPP}). Then $\tilde{u}$ solves
$$\frac{1}{\varepsilon}\tilde{ u}(y) =\varepsilon^2\varphi(\frac{y}{\varepsilon})+ (1-\varepsilon^2) \min_{player} \max_{market} \mathbf{E}[\frac{1}{\varepsilon}\tilde{u}( y+\varepsilon \Delta y)].$$
Here, too, if $\varepsilon \varphi(\frac{x}{\varepsilon}) = \varphi(x)$, then $\tilde{u}$ solves the scaled DPP (\ref{sGDPP}).

\subsection{Balanced Strategies}

The goal of this subsection is to prove that for finite, positive $\varepsilon$ an optimal strategy of the market can be achieved using 'balanced strategies' (to be explained in the lemma below). The argument for the following lemma  generalizes an argument in \cite{GPS}.

\begin{lemma}\label{balance}
Let $w(x_1, x_2,..., x_n)$ be a function satisfying the following properties:
\begin{enumerate}
\item $w$ is monotone nondecreasing in each $x_i$
\item $w(x_1+c,x_2+c,...,x_n+c) = w(x_1,x_2,...,x_n)+c$ \ for all $c \in \mathbb{R}$.
\end {enumerate}
Then, the market has at least one optimal strategy for
\begin{equation}\label{opt1}
 \min_{player} \max_{market} \mathbf{E}w(x+\varepsilon \Delta x)
\end{equation}
that is balanced in the sense that $$\mathbf{E}[v_i] =\mathbf{E}[v_j].$$ for all $i$ and $j$.
\end{lemma}
\begin{proof}
 Firstly,  we examine (\ref{opt1}), calling it `W'. Then,
\begin{eqnarray}
W &=& \min_{player} \max_{market} \mathbf{E}[w( x+\varepsilon \Delta x)].  \\
& =& \min_{player} \max_{market} \mathbf\sum_k \alpha_k\mathbf{E}_p[w( x+\varepsilon \vec v-\varepsilon\vec{1} v_k)]\\
& =& \min_{player} \max_{market} \left[ \mathbf\sum_k \alpha_k\mathbf{E}_p[w( x+\varepsilon \vec v)]-\varepsilon\sum_k \alpha_k\mathbf{E}_p[v_k] \right]\\
& =& \min_{player} \max_{market} \left[ \mathbf{E}_p[w(x+\varepsilon \vec v)]-\varepsilon\sum_k \alpha_k\mathbf{E}_p[v_k] \right]. \label{above}
\end{eqnarray}
Here $\alpha_k$ is the probability that the player follows expert $k$, $p$ is the market's probability distribution on the expert's outcomes, and $\vec 1 = (1,...,1).$
The equalities above follow by the definition of expected value, using translation invariance (i.e. property 2 above) and the fact that $\sum_k \alpha_k =1$.

Suppose there exists an optimal strategy $p$ for the market which is not balanced. We will construct an optimal strategy which is balanced.
Since the market is unbalanced, there exists an expert with a largest expected value, say it is expert $k$, i.e. $k=\mbox{argmax}_j({\mathbf{E}_p[v_j]})$. The expression \eqref{above} is a linear programming problem in $p$ and $\vec \alpha$, so $\min\max = \max\min$, i.e. the optimal strategies are unchanged if the player minimizes first. The player wants to minimize the second sum, because she has no influence over the first sum, so she may choose to follow expert $k$, i.e. she may choose $\alpha_k=1$. Pick an expert $i$ such that $\mathbf{E}_p[v_i] <\mathbf{E}_p[v_k]$; to simplify notation, suppose $i=1$ and we shall write $\mathbf{E}$ instead of $\mathbf{E}_p$. Then, consider the pair of market outcomes where the only difference is $v_1$'s value - 0 or 1. Observe that if the market increases the probability of a term where $v_1 =1$ at the expense of a term where $v_1=0$ , he increases $\mathbf{E}[w( x+\varepsilon \Delta x)]$, since, by monotonicity
\begin{align} w( x+\varepsilon (0,v_2,...v_n)-\varepsilon \vec{1}v_k)  \leq  w( x+\varepsilon (1,v_2,...v_n) -\varepsilon\vec1{ v_k}).
\end{align}
 By changing the probabilities of these two outcomes appropriately, the market obtains a strategy satisfying $\mathbf{E}[v_1] =\mathbf{E}[v_k]$ that is at least as good as the original one; note that the other expectations $\mathbf{E}[v_2], ...,\mathbf{E}[v_n]$ remain unchanged. Performing this operation for every $s$ such that $\mathbf{E}[v_s]<\mathbf{E}[v_k]$, we obtain a balanced strategy for the market which performs at least as well as the original optimal one.
\end{proof}

%%%%%%%%%%%%%%%%%%%%%%%%%%%%%%%%%%%%%%%%%%%%%%%%%%%%%%%%%%%%%%%%%%%

\section{Heuristic PDE Derivations\label{chap:two}}\label{DPP}

In  this section we use the DPP formulation to derive, at least heuristically, the associated PDEs. First we consider the geometric stopping case, then the finite horizon case.

\subsection{The PDE for Geometric Stopping Case}\label{PDEGS}

We  `derive' formally a limiting elliptic PDE. This derivation makes assumptions on the behavior of $u$, for example sufficient smoothness. For now the derivation is heuristic, but later on it will be justified, in the sense that we will prove that this game is a convergent numerical scheme for the PDE. Substituting the Taylor expansion of $u$ into the DPP \eqref{sGDPP} gives
\begin{eqnarray}\label{Tay}
%u(x) &=& \delta \varphi(x) +(1-\delta)\min_{player} \max_{market} \mathbf{E}[u(x+\varepsilon \Delta x)]   \nonumber \\
%%u(x) &=& \delta \varphi(x) +(1-\delta)\min_{player} \max_{market} \mathbf{E}[u +\varepsilon\langle\nabla u , \Delta x \rangle+\frac{\varepsilon^2}{2} \langle D^2u \cdot \Delta x, \Delta x\rangle + \nonumber \\
%%& & \hspace*{9.6cm} + O(\varepsilon^3)] \nonumber \\
u(x) &=&  \varphi(x) +\frac{(1-\varepsilon^2)\varepsilon}{\varepsilon^2} \min_{player} \max_{market} \mathbf{E}[\langle\nabla u , \Delta x \rangle+\frac{\varepsilon}{2} \langle D^2u \cdot \Delta x, \Delta x\rangle +  O(\varepsilon^2)].
\end{eqnarray}

As $\varepsilon \to 0$, the dominating term in the $\min\max$ is $\mathbf{E}[\langle\nabla u , \Delta x \rangle]$, so we focus on it:

\begin{eqnarray}
\mathbf{E}_{\vec{\alpha}, p}[\langle\nabla u , \Delta x \rangle]= \sum_{i=1}^{n} \alpha_i\mathbf{E}_{p}[\langle\nabla u , \vec v - v_i \cdot \vec 1\rangle] 
%&=& \sum_{i=1}^{n} \alpha_i \langle\nabla u , \mathbf{E}_{p}(\vec v -v_i\cdot \vec 1  )\rangle  \nonumber \\
%&=&  \sum_{i=1}^{n} \alpha_i  \sum_{k=1}^{n} \partial_k u\mathbf{E}_{p}(v_k -v_i ) \label{Vk}  \\
%&=&  \sum_{i=1}^{n} \alpha_i [ \sum_{k=1}^{n} \partial_k u\mathbf{E}_{p}v_k - ( \sum_{k=1}^{n} \partial_k u)\mathbf{E}_{p}v_i ]  \nonumber \\
%&=&  \sum_{i=1}^{n}[   \sum_{k=1}^{n} \partial_i u\alpha_k - \alpha_i  \sum_{k=1}^{n} \partial_k u]\mathbf{E}_{p}v_i    \nonumber \\
=  \sum_{i=1}^{n}[  \partial_i u - \alpha_i  \sum_{k=1}^{n} \partial_k u]\mathbf{E}_{p}[v_i].  \label{alpha}
\end{eqnarray}
The equality follows by linearity, inner product definition, rearranging, change of summation, and the fact that $\vec{\alpha}$ is a probability distribution.
We focus on the expression on the last line:
\begin{equation}\label{mm}
 \min_{player} \max_{market} \sum_{i=1}^{n}[  \partial_i u - \alpha_i  \sum_{k=1}^{n} \partial_k u]\mathbf{E}_{p}[v_i]. \end{equation}
This expression is a pair of dual linear programs in min max form, with variables $\vec\alpha$ and  $p$, which represent the player's and the market's probability distributions, respectively. As such,
$$\min_{player}\max_{market}=\max_{market}\min_{player}.$$
We prove in Subsection \ref{uwp} that $u_\varepsilon$ satisfies the following properties: monotonicity in each variable
and the translation property, i.e.
 \begin{equation}\label{tranp}
u(x_1+c,x_2+c,...,x_n+c) = u(x_1,x_2,...,x_n)+c.
\end{equation}
Later on, we will prove that $u_\varepsilon \to u$ and thus $u$ inherits those properties. Moreover, we are assuming for this heuristic discussion that $u$ is differentiable, so monotonicity turns into $\partial_ju\geq 0 ~~\forall~j$, whereas differentiating translation invariance, we obtain \begin{equation}\label{sum}
\sum_{i=1}^n\partial_i u=1.
\end{equation}
We claim that
\begin{enumerate}
\item the player's optimal strategy is $\alpha_i=\partial_iu$;
\item the market's optimal strategy is any probability distribution $p$ satisfying $\mathbf{E}_{p}[v_j] = \max_{k=1,...n}\mathbf{E}_{p}[v_k]$ for every $j$ such that $\partial_ju>0$; and
\item the value of the minmax in (\ref{mm}) is $0$.
\end{enumerate}
To prove 1, we observe that if  $\alpha_i=\partial_iu$, then (\ref{mm}) is $0$ for every choice of the market's strategy $p$. Suppose $\alpha_i \neq \partial_i u$. Since $\sum \alpha_i=1$, then there would exist a pair of indicies so that $(\partial_ju-\alpha_j )>0$ and $(\partial_k u-\alpha_k )<0.$ The market can take advantage of this and put all the weight into $v_j$, obtaining a positive contribution $(\partial_ju-\alpha_j )\mathbf{E}[v_j]>0,$ which is a worse outcome for the player. So the choice of  $\alpha_i=\partial_iu$ is superior to the player's other options.

To prove 2, we note that $\min_\alpha\sum(\partial_iu - \alpha_i)\mathbf{E}_{p}[v_i]$ attains the minimum when $\alpha_i \neq 0$ at summands where   $\mathbf{E}_{p}[v_j] = \max_{k=1,...n}\mathbf{E}_{p}[v_k]$.
Using $\sum \alpha_i=1$ and $\sum \partial_i u=1$, we obtain
$$\sum_{i=1}^n \partial_i u (\mathbf{E}_p[v_i]  - \max_{k=1,...n}\mathbf{E}_{p}[v_k]).$$
The maximal value the market can obtain is $0$, achieved when $$\mathbf{E}_p[v_i]  = \max_{k=1,...n}\mathbf{E}_{p}[v_k]$$ for all indices $i$ such that $\partial_i u>0$. If the market doesn't follow this strategy, the resulting value will be less than $0$. The proof of the claims is now complete.

Reviewing the preceding results, and assuming (as it seems natural) that $\partial_i u >0$ for all $i$, we see that the strategy of the player is fully determined:
\begin{eqnarray} \label{alpha}
 \alpha_i =  \partial _i u  \quad \fbox {market indifference condition,}
\end{eqnarray}
whereas the player influences (but doesn't fully determine) market's choices:
\begin{eqnarray} \label {Ep}
 \mathbf{E}_{p} [v_i] = \mathbf{E}_{p} [v_k]  \quad  \forall i,k \quad \fbox{balance condition}.
\end{eqnarray}
The optimal value of the $\min\max$ is $0$, so the $\varepsilon$ order term in the Taylor expansion vanishes. In order to obtain a PDE, we need to go to the second order of the Taylor expansion. We incorporate the knowledge of strategies of the player and the market by writing $\max_{balance}$ to indicate that $\alpha_i$ is determined by (\ref{alpha}) and $p$ is restricted to (\ref{Ep}).  Thus, we obtain:
\begin{eqnarray*}
u(x) &=&  \varphi(x) +\frac{1-\varepsilon^2}{2} \max_{balance} \mathbf{E}[  \langle D^2u \cdot \Delta x, \Delta x \rangle + O(\varepsilon)].
\end{eqnarray*}
%In order for the expected value term to have a meaningful contribution, we choose $ \delta=\varepsilon^2$, so that
%$$\lim_{\delta \to 0} \frac{1-\varepsilon^2}{\delta}  =1 .$$
%As $\varepsilon$ and $\delta$ go to $0$, we obtain the limiting equation
In the limit $\varepsilon \to 0$ we obtain the equation
\begin{eqnarray}
u(x) = \varphi(x) - \mathcal{L}(u), \label{elPDE}
\end{eqnarray}
where
\begin{eqnarray}
 \mathcal{L}(u) &:=&-\frac{1}{2}\max_{balance}\mathbf{E}[  \langle D^2u \cdot \Delta x, \Delta x\rangle ] \label{Lu}  \\
&=&-\frac{1}{2}\max_{balance}( \sum_{i, j, k} \alpha_k\partial^2_{ij}u\mathbf{E}[(v_i-v_k)(v_j-v_k)]).
\end{eqnarray}

\subsection{The PDE for the Finite Horizon Problem}

Returning to the time dependent problem, we observe a lot of similarities. Again we start by substituting the Taylor expansion of $w$ into the DPP \eqref{sTDPP}; this gives
\begin{eqnarray*}
% w(t,x) &=& \min_{player} \max_{market} \mathbf{E}[w(t+\varepsilon^2, x+\varepsilon \Delta x)] , \\
% w(t,x) &=& \min_{player} \max_{market} \mathbf{E}[w +\varepsilon^2w_t+ \varepsilon\langle\nabla w , \Delta x \rangle+\frac{\varepsilon^2}{2} \langle D^2w \cdot \Delta x, \Delta x\rangle + \\
%& &\hspace*{9cm}+O(\varepsilon^3)], \\
0 &=&  \min_{player} \max_{market} \mathbf{E}[\langle\nabla w , \Delta x \rangle+ \varepsilon(w_t +\frac{1}{2} \langle D^2w \cdot \Delta x, \Delta x\rangle) + O(\varepsilon^2)].
\end{eqnarray*}
Again, as $\varepsilon \to 0$, the dominating term is $\mathbf{E}_{\alpha, p}[\langle\nabla w , \Delta x \rangle]$.
The analysis of this term done in subsection \ref{PDEGS} applies here too. In particular,  the `market indifference' and the `balance' conditions are the same. This leaves the same restrictions over the $\min\max$ as in the previous case, hence the $\varepsilon^2$-order term has the same `balance' condition as in the previous case. This yields the limiting equation

\begin{eqnarray}
w_t - \mathcal{L}(w)=0, \label{parPDE}
\end{eqnarray}
for the operator $\mathcal{L}$ defined by (\ref{Lu}), with a final time condition
$$w(T,x) =\varphi(x).$$

\subsection{The Operator $ \mathcal{L}$} \label{el}

We need to understand the operator $\mathcal{L}(u)$.
Firstly, we investigate the expectation part. Let $p(v)$ be the probability of a particular vector $v \in \{0,1\}^n$, and let $\tilde v = \vec 1 - v$. Then,
$$ \mathbf{E}[(v_i-v_k)(v_j-v_k)] = \sum_{v \in \{0,1\}^n}(v_i-v_k)(v_j-v_k)p(v) = \sum_{v \in \{0,1\}^n}\mathbbm{1}_{v_i =v_j \neq v_k}(v)p(v),$$
where $\mathbbm{1}$ is the indicator function.

Substituting in $\mathcal{L}(u)$, we obtain
\begin{eqnarray} \label{sLu}
 \mathcal{L}(u) &=&- \frac{1}{2} \max_{balance}\sum_{i, j, k}  \sum_{v \in \{0,1\}^n}(\alpha_k\partial^2_{ij}u\mathbbm{1}_{v_i = v_j \neq v_k})p(v) \nonumber \\
&=&-\frac{1}{4} \max_{balance} \sum_{v \in \{0,1\}^n}\sum_{i, j, k} (\alpha_k\partial^2_{ij}u\mathbbm{1}_{v_i =v_j \neq v_k})(p(v)+p(\tilde v)),
\end{eqnarray}
since $\mathbbm{1}_{v_i =v_j \neq v_k}$ takes the same value for $v$ and for $\tilde{v} = \vec 1 - v$ (for any triplet $i,j,k$). In view of (\ref{sLu}) we can treat $p$ as a probability distribution on pairs of complementary strategies. The restriction of `balance' can be ignored, since if we choose $v$ and $\tilde v$ to have  the same probability for every $v$, then
$$\mathbf{E}_p[v_i] =\frac{1}{2}~~~~~~ \forall i.$$

Recall that equation (\ref{sum}) holds:
\begin{equation}
\sum_{k=1}^n \partial_k u =1 \nonumber.
\end{equation}
 For any fixed $v \in \{0,1\}^n$ we write this as $$\sum_{i: v_i=0}\partial_{i} u = 1  - \sum_{i: v_i=1}\partial_{i} u $$ and differentiate again to get
\begin{eqnarray*}
 (\sum_{j: v_j=0}\partial_{j})(\sum_{i: v_i=0}\partial_{i} u )=- \sum_{j: v_j=0}\partial_{j}( \sum_{i: v_i=1}\partial_{i} u) =  \sum_{i: v_i=1}\partial_{i}(- \sum_{j: v_j=0}\partial_{j} u) =\sum_{i: v_i=1}\partial_{i}(\sum_{j: v_j=1}\partial_{j} u).
\end{eqnarray*}
Thus we obtain the equality
\begin{equation}\label{Sq}
(\sum_{i: v_i=0}\partial_{i})^2 u=(\sum_{i: v_i=1}\partial_{i})^2 u ,
\end{equation}
which we will use in %the third line of
the following calculation of the sum on the right hand side of (\ref{sLu}). For any fixed $v \in \{0,1\}^n$, let $$S=\sum_{i, j, k} (\alpha_k\partial^2_{ij}u\mathbbm{1}_{v_i =v_j \neq v_k}).$$ Then
\begin{eqnarray*}
S &=&
% \sum_{k: v_k=1 } \alpha_k\sum_{i,j: v_i=v_j=0}\partial^2_{ij}u +\sum_{k: v_k=0 } \alpha_k \sum_{i,j: v_i=v_j=1} \partial^2_{ij}u \\
%&=&( \sum_{k: v_k=1 } \alpha_k)(\sum_{i: v_i=0}\partial_{i})^2 u + ( \sum_{k: v_k=0 } \alpha_k)(\sum_{i: v_i=1}\partial_{i})^2 u \\
%&=& ( \sum_{k: v_k=1 } \alpha_k)(\sum_{i: v_i=1}\partial_{i})^2 u + ( \sum_{k: v_k=0 } \alpha_k)(\sum_{i: v_i=1}\partial_{i})^2 u\\ &=&
( \sum_{k } \alpha_k)(\sum_{i: v_i=1}\partial_{i})^2 u = (\sum_{i: v_i=1}\partial_{i})^2 u
\end{eqnarray*}
(by rearrangement of derivatives, %. The last equality is obtained by
combining terms, and observing that the sum of $\alpha_k$ equals $1$.) Returning now to \eqref{sLu}, we have
\begin{eqnarray}
 \mathcal{L}(u) &=&- \frac{1}{4}\max_{balance}\sum_{v\in\{0,1\}^n} (p(v)+p(\tilde v))(\sum_{i: v_i=1}\partial_{i})^2 u\\
\nonumber
&=&- \frac{1}{2}\max_{v\in\{0,1\}^n} \langle D^2u\cdot v, v\rangle.
\end{eqnarray}
For the second line above we used that the probabilities $p$ sum up to 1, so the maximum linear combination, weighted by those probabilities, is achieved by assigning all the weight on the largest term.

In conclusion, the elliptic PDE \eqref{elPDE} is
\begin{equation*} %\label{EE}
u(x) = \varphi(x) +  \frac{1}{2}\max_{v\in\{0,1\}^n} \langle D^2u\cdot v, v\rangle, %\tag{EE}
\end{equation*}
and the parabolic PDE \eqref{parPDE} is
\begin{equation*} %\label{PE}
w_t(t,x) + \frac{1}{2}\max_{v\in\{0,1\}^n} \langle D^2w(t,x)\cdot v, v\rangle,
 = 0. %\tag{PE}
\end{equation*}
as announced earlier in (\ref{Equa}) and (\ref{parE}).

The justification of our heuristic calculation, to be presented in Section \ref{CV}, relies on the fact that our operator $ \mathcal{L}$ is degenerate elliptic. We check this now. Recall that, by definition, an operator $\mathcal{L}(u, D^2u)$ is degenerate elliptic if
$$\mathcal{L}(u, M_1) \leq \mathcal{L}(u, M_2)$$
when $M_1-M_2$ is non-negative, that is $M_1-M_2 \geq 0$ as matrices.

\begin{lemma}\label{dege}
The operator
$$\mathcal{L}(u)=-\frac{1}{2}\max_{v\in\{0,1\}^n}\langle D^2u\cdot v, v\rangle$$
is degenerate elliptic.
\end{lemma}
\begin{proof}
Let $M_1-M_2 \geq 0$. Then, for any $v$ we have $\langle M_1\cdot v, v\rangle \geq \langle M_2\cdot v, v\rangle$. We take the maximum over the set of vectors $v$ such that ${v\in\{0,1\}^n}$: first on the left side, then on the right side, obtaining
$$\max_{v\in\{0,1\}^n}\langle M_1\cdot v, v\rangle \geq \max_{v\in\{0,1\}^n}\langle M_2\cdot v, v\rangle.$$ Finally, we multiply by $-1/2$ to obtain the desired inequality $$\mathcal{L}(u, M_1) \leq \mathcal{L}(u, M_2).$$
\end{proof}

\subsection{Optimal strategies} \label{OptStr}

A remaining question is what the PDEs tell us about the optimal strategies for the player and the market. The answer lies (formally, at least) in the preceding calculation. Consider the elliptic PDE and suppose its solution is known and  $C^2$. Suppose the vector of regrets so far is $x$. Then the best move of the player is to follow expert $i$ with probability
$$\alpha_i = \partial _i u(x) .$$
In turn, the market looks for a $v \in \{0,1 \}^n$ (and its complement $\vec1-v$) that saturates the maximum in
\begin{equation}
\mathcal{L}(u) = -\frac{1}{2}\max_{v\in\{0,1\}^n}\langle D^2u\cdot v, v\rangle.
\end{equation}
Observe that by (\ref{sLu}), $v$ saturates the maximum precisely when $1-v$ saturates the maximum.
Having found $v$, the market's optimal strategy is this:  with probability  $1/2$ advance the experts such that $v_i=1$, and  with probability $1/2$ advance the rest of the experts, i.e. those for which $v_i=0$.
If $\max_{v\in\{0,1\}^n}\langle D^2u\cdot v, v\rangle$  is achieved for more than one pair of vectors  $v$ and its complement $\vec{1}-v$, then the market's strategy is not unique.

\subsection{Comparison with paper \cite{GPS} by Gravin, Peres, Sivan}\label{GPScomp}

Our work is closely related to paper \cite{GPS} by Gravin, Peres, and Sivan. Briefly: this paper and \cite{GPS} look at the same problem through different lenses. The fundamental difference is that we study a natural continuum limit, while they focus on the problem in its original discrete-time form. This leads to differences with respect to \cite{GPS} in both the character of our results and the methods used to demonstrate them. Our rigorous results are mainly concerned with the value function, which we characterize as the unique viscosity solution of an appropriate PDE problem; in deriving these results, we also obtain heuristic guidance about how the optimal strategies are related to the solution of the PDE. In \cite{GPS}, by contrast, no PDE is studied; instead, the value of the game is studied using methods from random walks, combined with what an optimal control theorist would call ``verification arguments.''  Of course \cite{GPS} also studies the form of the optimal strategies, and its conclusions are similar to ours. However our continuum viewpoint offers a different perspective, in which the main features of the optimal strategies are understood by considering a linear programming problem.

Another distinction from \cite{GPS} is the choice of how to measure ``regret.'' Our methods permit treatment of the continuum problem with a relatively broad class of measures $\varphi$ of regret: if $x_j$ is the player's regret with respect to the $j$th expert, we require mainly that $\varphi(x_1,\ldots,x_N)$ be increasing in each variable, satisfy $\varphi(x_1+c,\ldots, x_N+c) = \varphi(x_1,\ldots,x_N) + c$,  and have linear growth at infinity. The paper \cite{GPS}, by contrast, focuses exclusively on the classic measure $\varphi(x) = \max_{j=1}^N x_j$ (i.e. the player's shortfall compared to the best-performing expert).

There are, of course, many similarities and parallels between our work and \cite{GPS}. In fact, our work began when we read \cite{GPS} and realized that a continuum perspective might be of interest. A particular parallel is worth noting: our exact solution of the geometric stopping problem with 3 players and objective $\varphi(x) = \max_{j=1}^3 x_j$ is the continuum analogue of a result proved in the discrete setting in \cite{GPS}. (We found it by looking at the optimal strategies identified in  \cite{GPS} and considering their continuum analogues.)

 %%%%%%%%%%%%%%%%%%%%%%%%%%%%%%%%%%%%%%%%%%%%%%%%%%%%%%%%%%%%%%%%%%%%

\section{The Games as  Numerical Schemes for the PDEs\label{chap:three}}\label{GS}

This section discusses the discrete solutions $u_\varepsilon$ and $w_\varepsilon$. Concerning the former: even the existence of $u_\varepsilon$ is not immediately obvious. We prove it (and obtain an estimate that is uniform in $\varepsilon$) by representing the time-independent dynamic programming principle as a "numerical scheme for the PDE \eqref{Equa}" similar to those discussed by eg Oberman's paper \cite{Oberman}.

In this section we represent the time-independent discrete problem as a numerical scheme  $\mathcal{F}_\varepsilon$ for the elliptic PDE (\ref{Equa}). Throughout this section we follow the setup of Oberman's paper \cite{Oberman}  in discussing the scheme and showing that the DPP has a unique solution. In particular, all the definitions in this section are from  \cite{Oberman}, as well as adapted theorem statements and proofs.
Our treatment differs from \cite{Oberman} in that we work with a scheme which is continuous, not discrete, in space.

This section also discusses the solution $w_\varepsilon$ of the  finite horizon problem. There the existence and uniqueness of $w_\varepsilon$ are easily established, but we need to prove uniform estimates as $\varepsilon \to 0$.
\subsection{Definitions of $\mathcal{F}_{\varepsilon}$, $S_\rho$, and Basic Properties}   \label{DFS}

In writing the DPP, one considers a point $x$ and all its `neighbors', which are of the form $x +\varepsilon(v-v_k\vec{1})$;  we write $N(x)$ for the collection of all such neighbors as $v$ ranges over $\{ 0,1\}^n$. We order the neighbors in some order, say increasing if $(v,v_k)$ were written in binary as a $n+1$-letter word, to obtain neighbors $x_{v,v_k}=x +\varepsilon(v-v_k\vec{1})$, where $(v,v_k)=0,1, 2, ... ,  2^{n+1}-1$; altogether there are $N=2^{n+1}$ neighbors, where $n$ is the number of experts. From now on, we write $u(x_j)=u_j$. In particular, we use the convention that $u_0= u(x)$.

We consider the solution to the geometric stopping problem, which we rearrange by subtracting $(1-\varepsilon^2)u(x)$, combining all terms on one side, and dividing by $\varepsilon^2$:
\begin{eqnarray*}
u(x) &=& \varepsilon^2 \varphi(x) +(1-\varepsilon^2)\min_{player} \max_{market} \mathbf{E}[u(x+\varepsilon \Delta x)] ,
\end{eqnarray*}
so
\begin{eqnarray*}
0&=& u(x) - \varphi(x) -\frac{1-\varepsilon^2}{\varepsilon^2}\min_{player}\max_{market}\mathbf{E}[u(x+\varepsilon\Delta x)-u(x)].
\end{eqnarray*}

Inspired by this rearrangement of the geometric DPP, we define the time-independent approximation scheme as $ \mathcal{F}_\varepsilon[u]=0$, where
{\small
\begin{eqnarray}
 \mathcal{F}_\varepsilon[u](x)&:=& u(x) - \varphi(x) -\frac{1-\varepsilon^2}{\varepsilon^2}\min_{player}\max_{market}\mathbf{E}[u(x+\varepsilon\Delta x)-u(x)] \\ \nonumber
 &=& u(x) - \varphi(x) -\frac{1-\varepsilon^2}{\varepsilon^2}\min_{player}\max_{market}\sum p(v)\alpha_k[u_j -u_0]\\
&=&  u - \varphi +\frac{1-\varepsilon^2}{\varepsilon^2}\max_{player}\min_{market}\sum p(v)\alpha_k[u_0 -u_j]\\ \nonumber
&=& F_\varepsilon^x(u, u_0-u_j).
\end{eqnarray}
}

Evidently, for any fixed $x$ the value of
$$\mathcal{F}_\varepsilon[u](x):=F_\varepsilon^x(u_0, u_0-u_j),$$
depends only on the values $u$ at $x$, and its neighbors  $x + \varepsilon\Delta x.$
In $F_\varepsilon^x(\cdot, \cdot)$ the first argument refers to the function $u(x)$ before $\varphi$, and the subsequent arguments $u_0 - u_j$, $j=0,1, ... , N-1$ refer to the finite differences $u_0-u_j$ in the expected value terms.

We will prove that the scheme has a number of properties, whose analogues can be found in \cite{Oberman}:
\begin{definition}
The scheme $\mathcal{F}_\varepsilon$ is proper if there exists $\delta > 0$ such that for all $x, y \in \RR^{N}$ and $x_0, y_0 \in \RR$,
$$x_0 \leq y_0~ \mbox{ implies }~ F_\varepsilon^x(x_0, y)-  F_\varepsilon^x(y_0, y) \leq \delta(x_0-y_0). $$
\end{definition}

\begin{definition}
The scheme $\mathcal{F}_\varepsilon$ is degenerate elliptic if  the map
 $$F_\varepsilon^x(u_0, u_0-u_j)$$ is non-decreasing in each variable $u_0, u_0-u_j$ for all $j=0,1, 2, ... , 2^{n+1}-1$.
\end{definition}

\begin{definition}
The finite difference scheme $F_{\varepsilon}$ is Lipschitz continuous if there exists a constant $K$ such that for all $z,y \in \RR^{N+1}$,
\begin{equation}
 |F_\varepsilon(z)-F_\varepsilon(y)| \leq K ||z-y||_\infty.
\end{equation}
\end{definition}

\begin{lemma}
The scheme $\mathcal{F}_\varepsilon$ is proper and  degenerate elliptic.
\end{lemma}
\begin{proof}
The scheme is proper as $F_\varepsilon(x_0, y)-  F_\varepsilon(y_0, y)=x_0 -y_0 $.\\
The operator $ \frac{1-\varepsilon^2}{\varepsilon^2}\max_{\alpha}\min_{p}\sum p(v)\alpha_k[u -u_j]$ is degenerate elliptic as a $\max\min$ of a positive linear combination of its $u$-differences. Therefore, the scheme $F_\varepsilon^x(u, u-u_j)$ is degenerate elliptic: it is a sum of the function $u $, the function $-\varphi$, and a degenerate elliptic operator.
\end{proof}

%\section{Lipschitz Continuity}
\begin{lemma}
The scheme $\mathcal{F}_\varepsilon$ is Lipschitz continuous with $K =1+  (1-\varepsilon^2)/\varepsilon^2.$
\end{lemma}
\begin{proof}
Firstly, observe that the sum of two Lipschitz continuous schemes is Lipschitz continuous. Since $u-\varphi$ is  Lipschitz continuous with a constant 1, we only need to find a Lipschitz constant C for the $\max\min$ part of the scheme; then $K=1 +C(1-\varepsilon^2)/\varepsilon^2$.

Define $F_{p,\alpha}[\tilde{u}]=\sum_{j(v,v_k)} p(v)\alpha_k\tilde {u}_j$. Observe that $F_{p,\alpha}$ is a linear combination of its independent variables $\tilde{u}_j$ with weights that are non-negative and sum up to $1$, as the non-negative weights come from an expectation. Then, $F_{p,\alpha}$ is Lipschitz continuous with constant $1$. For any admissible vectors $u, w$, we get the following sequence of inequalities:
\begin{eqnarray*}
 F_{p,\alpha}[u] &\leq& F_{p,\alpha}[w] + |u-w|_\infty,  \\
\max_{\alpha} F_{p,\alpha}[u] &\leq&\max_\beta F_{p,\beta}[w] + |u-w|_\infty,\\
\min_p\max_{\alpha} F_{p,\alpha}[u] &\leq&\min_\rho\max_\beta F_{\rho,\beta}[w] + |u-w|_\infty.
\end{eqnarray*}
The same equality holds, of course, with $u$ and $w$ switched. Hence, $\min_p \max_\alpha F_{p,\alpha}[\tilde{u}]$ is Lipschitz continuous with constant $1$. This means that the overall Lipschitz constant is  $K =1+  (1-\varepsilon^2)/\varepsilon^2$.
\end{proof}

 We introduce some notation for the next lemma. Given $u, w \in \RR^M$, define $u \lor w = \max (u, w)$, $u^+ = \max(u,0)$, $u^- =\min(u,0)$. The following lemma is found in \cite{Oberman}.
\begin{lemma}
(ordered Lipschitz continuity property) Let $\mathcal{F}_\varepsilon$ be a Lipschitz continuous, degenerate elliptic scheme with Lipschitz constant $K$. Then for any $y,z \in \mathbb{R}^{N+1}$ we have
\begin{equation}
 -K||(z-y)^-||_\infty \leq F_\varepsilon (z) - F_\varepsilon (y) \leq K||(z-y)^+||_\infty .
\end{equation}
\end{lemma}

%\textbf{Removed the proof}
%\textbf{BeginRemove}
%\begin{proof}
%We use the definition of degenerate ellipticity and the Lipschitz continuity property in the following way:
%$$ F_\varepsilon(z)-F_\varepsilon(x) \leq F_\varepsilon(z \lor y) - F_\varepsilon(y)\leq K||z\lor y - y ||_\infty = K||(z - y)^+ ||_\infty  .$$
%Similarly,
%$$ F_\varepsilon(z)-F_\varepsilon(x) \geq -(F_\varepsilon(z \lor y) - F_\varepsilon(y))\geq -K||-(z\lor y) - y ||_\infty =- K||(z - y)^- ||_\infty  .$$
%\end{proof}
%\textbf{EndRemove}

\subsection{The Euler Map}\label{EM}
We define the Euler map associated to our scheme $\mathcal{F}_\varepsilon[u]$.
\begin{definition}
 For $\rho >0$, define the explicit Euler map $S_\rho$ by
\begin{equation}
 S_\rho (u) = u - \rho \mathcal{F}_\varepsilon[u].
\end{equation}
\end{definition}

Intuitively: the scheme $\mathcal{F}_\varepsilon[u_\varepsilon]=0$ is a numerical approximation of an elliptic PDE, and the map $ u \mapsto S_\rho(u)$  is the time step map  for an explicit discretization of the associated parabolic equation.
The following theorem and its proof are found in \cite{Oberman}.
\begin{theorem}
Fix $\rho$ such that $\rho K <0.5.$ Then, the Euler map is monotone.
\end{theorem}
\begin{proof}
Suppose $u \leq w$. Then,
\begin{eqnarray*}
S_\rho (u)-S_\rho (w)&=& u_0 - w_0 +\rho(F_\varepsilon(w_0, w_0-w_j)- F_\varepsilon(u_0, u_0-u_j)), \\
&\leq& u_0 - w_0 +\rho K ||(w_0-u_0, w_0-w_j -u_0+u_j)^+ ||_\infty,\\
&\leq& (u_0 - w_0) (1-\rho K) \leq 0.
\end{eqnarray*}

The first inequality follows from the ordered Lipschitz continuity property. The second inequality follows from  $u \leq w$, and the last one from the assumption of the theorem. This establishes monotonicity.
\end{proof}

\subsection{Properties of $\tilde{\varphi}$}\label{smooth}
We work with $\varphi$ - a measure of regret and a Lipschitz continuous function which also satisfies properties (1.2 -1.6). One example of such a function is the classical
$$\varphi(x)=\max_k \{x_k\},$$ which has discontinuous first derivatives, so we don't want to assume that $\varphi$ is smooth. We will need a smoothed version of $\varphi$. We define it using
 a mollifier $\eta$, defined as:
\[
\eta(x)  = \begin{cases}c_\eta e^{-\frac{1}{1-|x|^2}} ~~~~~~~~~~\text{ if}~~ |x|<1 \\ 0~~~~~~~~~~~~~~~~~~~~~\text{ if} ~~|x|\geq 1,  \end{cases}
\]
where the constant $c_\eta$ is chosen so that $\eta$ integrates to $1.$ Our smoothed version of $\varphi$ is
\begin{equation}
\tilde{\varphi} = \varphi*\eta.
\end{equation}
The following specific properties of $\tilde{\varphi}$ are easily verified:
\begin{align}
\tilde{\varphi}~ & \text{is}~ C^\infty, \text{with uniformly bounded derivatives of order} ~k~ \text{for any}~ k \geq 1, \label{Prop1} \\
\tilde{\varphi}~ & \text{is uniformly 'close to'}~\varphi ~ \text{in the sense that } ~ |\varphi-\tilde{\varphi}|_\infty \leq K~\nonumber \\
 &\text{for} ~K=\text{ the constant of the global Lipschitz bound}, \label{Prop2}\\
\tilde{\varphi}~ & \text{is monotone in each variable},\label{Prop3}\\
\tilde{\varphi}~ & \text{has the translation property (\ref{tranp}).}\label{Prop4}
\end{align}

%\textbf{BeginRemove}
%For property (\ref{Prop1}) we use
%\begin{equation}
%\nabla \tilde{\varphi}(x) =\int\nabla \varphi (x-y) \eta(y)dy
%\end{equation}
%to obtain $ |\nabla\tilde{\varphi}|_\infty \leq |\nabla \varphi|_\infty$. Similarly, for $k \geq 2$
%\begin{equation}
%\nabla^k \tilde{\varphi}(x) =\int\nabla \varphi (x-y)\nabla^{k-1} \eta(y)dy
%\end{equation}
%to show that $|\nabla^k \tilde{\varphi}|_\infty \leq C_k|\nabla \varphi|_\infty.$
%
%In property (\ref{Prop2}) we use that $\varphi$ is Lipschitz continuous with constant $m$:
%\begin{equation}
%|\tilde{\varphi}-\varphi|_\infty=|\varphi*\eta-\varphi|_\infty\leq \max_{|y-x|\leq 1}|\varphi(y)-\varphi(x)| \leq \max_{|y-x|\leq 1}m|y-x|\leq m.
%\end{equation}
%
%For property (\ref{Prop3}) we use that $\varphi$ is monotone: then for any $x_1\leq \tilde{x}_1$
%{\small
%\begin{equation}
% \tilde{\varphi}(x_1, X) =\int \varphi (x_1-y_1,X-Y) \eta(y)dy \leq \int \varphi (\tilde{x}_1-y_1,X-Y) \eta(y)dy= \tilde{\varphi}(\tilde{x}_1, X).
%\end{equation}
%}
%Lastly, for the translation property (\ref{Prop4})
%\begin{eqnarray}
% \tilde{\varphi}(x +c\vec{1}) =\int \varphi (x+c\vec{1}-y) \eta(y)dy = \int( \varphi (x-y)+c) \eta(y)dy=\nonumber \\
%\tilde{\varphi}(x)+ c
%\end{eqnarray}
%because $\int \eta =1.$
%\textbf{EndRemove}

Now, we estimate the expectation term, when $\tilde{\varphi}$ replaces $u$. In order to do so, we use its Taylor expansion:
\begin{eqnarray*}
E&:=&| \min_{player}\max_{market}\mathbf{E}[\tilde{\varphi}(x+\varepsilon\Delta x)-\tilde{\varphi}(x)]  |\\
%&\leq&  |\min_{player}\max_{market}\mathbf{E}[-\tilde{\varphi}+\tilde{\varphi}+\varepsilon\langle\nabla \tilde{\varphi}, \Delta x \rangle+\frac{\varepsilon^2}{2} \langle D^2\tilde{\varphi} \cdot \Delta x, \Delta x\rangle] +O(\varepsilon^3)|\\
%&=& \Big{|}\min_{player}\max_{market}\Big{[}\varepsilon\mathbf{E}[\langle\nabla \tilde{\varphi}, \Delta x \rangle]+\frac{\varepsilon^2}{2}\mathbf{E}[\langle D^2\tilde{\varphi} \cdot \Delta x, \Delta x\rangle]+O(\varepsilon^3)\Big{]}\Big{|}\\
%&=&
&\leq& \Big{|}\min_{player}\max_{market}\Big{[}\varepsilon\sum_{k=1}^n\langle\nabla \tilde{\varphi}_k,\mathbf{E}[ \Delta x_k ]\rangle+\frac{\varepsilon^2}{2}\mathbf{E}[\langle D^2\tilde{\varphi} \cdot \Delta x, \Delta x\rangle]+O(\varepsilon^3)\Big{]}\Big{|}.
\end{eqnarray*}
Let us focus on the $\varepsilon$-order factor. Because of Lemma \ref{balance} it is sufficient to consider balanced strategies for the market. For such strategies we have
$$ \mathbf{E}[\Delta x_k]=\sum_{i=1}^n \alpha_i \mathbf{E}_p[v_k-v_i] =0.$$
So the $\varepsilon$ order term is  $0$. Then, we can bound the term with $\varepsilon^2$ (using the uniform bound on $\nabla^2\tilde{\varphi}$), obtaining
\begin{equation}
E \leq| \min_{player}\max_{market}\mathbf{E}[\tilde{\varphi}(x+\varepsilon\Delta x)-\tilde{\varphi}(x)]  | \leq C\varepsilon^2.\label{Emin} \end{equation}
We use this result in the following lemma.
\begin{lemma}
The function $\tilde{\varphi}$ is an almost-solution to the scheme, i.e. $|\mathcal{F}_\varepsilon[\tilde{\varphi}]|\leq K_1$ for some constant $K_1$, independent of the small parameter $\varepsilon$.
\end{lemma}
\begin{proof}
Let us bound the absolute value of the scheme at $\tilde{\varphi}$. We use the preceding estimate for $E$
\begin{eqnarray*}
|\mathcal{F}_\varepsilon[\tilde{\varphi}]| &=&|\tilde{\varphi}-\varphi + \frac{1-\varepsilon^2}{\varepsilon^2}\min_{player}\max_{market}\mathbf{E}[\tilde{\varphi}(x+\varepsilon\Delta x)-\tilde{\varphi}(x)]  | \\
&\leq& |\tilde{\varphi}-\varphi| +\frac{1-\varepsilon^2}{\varepsilon^2}|\min_{player}\max_{market}\mathbf{E}[\tilde{\varphi}(x+\varepsilon\Delta x)-\tilde{\varphi}(x)]  | \\
&\leq& K+ \frac{1-\varepsilon^2}{\varepsilon^2}C\varepsilon^2 \\
&\leq& K+C(1-\varepsilon^2).
\end{eqnarray*}
This has the form we want:
 $$|\mathcal{F}_\varepsilon[\tilde{\varphi}] |\leq K_1.$$
\end{proof}

\subsection{Existence and Uniqueness of a Solution of $\mathcal{F}_\varepsilon$ }\label{uwp}

\begin{theorem} \label{con}
Fix $\rho$ so that $\rho K <0.5.$ Then, for some $M>0$ (independent of $\varepsilon$) the Euler map is a strict contraction in the sup norm on a ball of size $M$, centered at $\tilde{\varphi}$.
\end{theorem}
The proof of Theorem \ref{con} is parallel to the proof of Theorem 7 from \cite{Oberman}.
\medskip

We now present the main result of this subsection:
\begin{theorem}\label{Euler}
The scheme $\mathcal{F}_\varepsilon$ has a unique solution $u_\varepsilon$ in the class of functions $u$ such that $u-\tilde{\varphi}$ is uniformly bounded on $\RR^n$. Moreover the solution  $u_\varepsilon$ has the following properties:
\begin{enumerate}
\item There is a constant $M$ such that $|u_\varepsilon - \tilde{\varphi}|\leq M ~~~~\forall  x \in\RR^n$.
\item The function $u_\varepsilon$  is monotone nondecreasing in each variable $x_j$.
\item The function $u_\varepsilon$ has the translation property, i.e.
$$u_\varepsilon(x_1+c,x_2+c,...,x_n+c) = u_\varepsilon(x_1,x_2,...,x_n)+c.$$
\end{enumerate}
\end{theorem}
\begin{proof}
Observe that  $|u-\tilde{\varphi}|$ is bounded if and only if $|u-\varphi|$ is bounded. By  theorem \ref{con}, $S_\rho$ is a strict contraction (with the maximum norm) on the set of functions $\{u\, :\,|u-\tilde{\varphi}| \leq M\}$ for some $M$. Here $M$ is a constant independent of $\varepsilon$.  We realize that the assertion holds for all sufficiently large $M$, independent of $\varepsilon$.  By the contraction mapping theorem,  $S_\rho$ has a unique fixed point in the set above. The solution is obtained by iterating  (with $\rho$ sufficiently small) starting from arbitrary initial data in the ball about $\tilde{\varphi}$ with radius $M$. Being a fixed point, i.e. satisfying $U=U-\rho\mathcal{F}_\varepsilon[U]$, is equivalent to satisfying  $\mathcal{F}_\varepsilon[U]=0$,  which is equivalent to satisfying the geometric dynamic programming principle. Therefore we see that the fixed point of  $S_\rho$, namely $u_\varepsilon$, is the desired solution of the scheme.

We already addressed the growth of our solution. As for monotonicity and translation invariance, we present the proofs in lemmas \ref{mono2} and \ref{tran2} below.
\end{proof}

\begin{lemma}\label{sym1}
The solution $u_\varepsilon$ is symmetric, i.e. we can switch the values of every pair of spatial coordinates without changing the function's value:
\begin{equation}
u_\varepsilon(x_1,x_2,...,x_n) = u_\varepsilon(x_2,x_1,...,x_n).
\end{equation}
\end{lemma}
\begin{proof}
This is a consequence of uniqueness but for clarity we prove it using induction.

For simplicity of notation we prove the above claim for $x_1$ and $x_2$.
The proof goes by induction on the iterates of the Euler map $S_\rho$.
Consider any $\varepsilon >0$, small.
Firstly, $\varphi$ is symmetric, i.e. $\varphi(x_1,x_2,...,x_n) = \varphi(x_2,x_1,...,x_n).$
Next, suppose $\psi$ is symmetric, i.e.  $\psi(x_1,x_2,...,x_n) = \psi(x_2,x_1,...,x_n).$ Then we observe that the function
$$f(x)=\psi(x)(1-\frac{\rho}{\varepsilon^2}) +\rho\varphi(x)+\frac{\rho}{\varepsilon^2}(1-\varepsilon^2)\min_{player}\max_{market}\mathbf{E}\psi(x+\varepsilon\Delta x) $$
is also symmetric since experts $1$ and $2$ have symmetric roles in the game. Observe that the function $f$ above is simply equal to $S_\rho$:
$$S_\rho = u(x) - \rho[ u(x) - \varphi(x) -\frac{1-\varepsilon^2}{\varepsilon^2}\min_{player}\max_{market}\mathbf{E}[u(x+\varepsilon\Delta x)-u(x)]] = f(x).$$
Thus if $\psi$ is symmetric, then $S_\rho(\psi)$ is symmetric.
So we iterate applying the Euler map $S_\rho$, starting from the symmetric $\varphi$.
By theorem \ref{Euler}, the iterates of the Euler map converge to the unique solution $u_\varepsilon$ to $\mathcal{F}_\varepsilon$.
We pass the symmetry property through the limit, obtaining that $u_\varepsilon$ is symmetric.
\end{proof}

\begin{lemma}\label{mono2}
The solution $u_\varepsilon$ is monotone, i.e. if $\tilde{x}_1\geq x_1$, then
\begin{equation}
u_\varepsilon(\tilde{x}_1,x_2,...,x_n) \geq u_\varepsilon(x_1,x_2,...,x_n).
\end{equation}
This property follows for every coordinate, as the proof for all other coordinates is identical.
\end{lemma}
\begin{proof}
The argument here is similar to the one in Lemma \ref{sym1}.

\end{proof}
%\textbf{BeginRemove}
%Since $u_\varepsilon=lim_{n \to \infty}S_\rho^n[\tilde{\phi}]$, all we have to show is that if $\psi$ is monotone in $x$ , then $S_\rho[\psi]$ is also monotone.
%Consider a function $\psi$, which is increasing in its first argument, i.e.  if $\tilde{x}_1 \geq x_1$, then
%$$\psi(\tilde{x}_1,x_2,...,x_n) \geq \psi(x_1,x_2,...,x_n).$$
%For any fixed pair of market and player strategies, $\mathbf{E}\psi(x+\varepsilon\Delta x)$ is an increasing function of $x_1$. It follows that $\min_{player}\max_{market}\mathbf{E}\psi(x+\varepsilon\Delta x)$ is an increasing function of $x_1$. Therefore
%$$f(x)=\psi(x)(1-\frac{\rho}{\varepsilon^2}) +\rho\varphi(x)+\frac{\rho}{\varepsilon^2}(1-\varepsilon^2)\min_{player}\max_{market}\mathbf{E}\psi(x+\varepsilon\Delta x) $$ is also increasing in its first coordinate. Observe that the function $f$ above is simply equal to $S_\rho[\psi]$:
%{\small
%$$S_\rho[\psi] = \psi(x) - \rho[ \psi(x) - \varphi(x) -\frac{1-\varepsilon^2}{\varepsilon^2}\min_{player}\max_{market}\mathbf{E}[\psi(x+\varepsilon\Delta x)-\psi(x)]] = f(x).$$
%}
%Thus the function $\psi$  increasing in its first coordinate implies $S_\rho[\psi]$ is also increasing in its first coordinate.
%So we iteratively apply the Euler map $S_\rho$, starting from the monotone $\varphi$.
%\textbf{EndRemove}

\begin{lemma}\label{tran2}
The solution $u_\varepsilon$  has the following property: for any $c \in \RR$, and any  $(x_1,x_2,...,x_n)\in \RR^n$
\begin{equation}\label{move}
u_\varepsilon(x_1+c,x_2+c,...,x_n+c) = u_\varepsilon(x_1,x_2,...,x_n)+c.
\end{equation}
\end{lemma}
\begin{proof}
The argument here is similar to the one in Lemma \ref{sym1}.
\end{proof}

\subsection{Growth and Qualitative Behavior of the Solutions to the Finite Horizon Problem}

In the previous subsection, we showed that the solution to the discrete geometric stopping problem has at most linear growth as $|x| \to \infty$.  We now show that the discrete solution of the finite horizon problem also has at most linear growth in $x$. This is achieved by the following theorem:

\begin{theorem}\label{Th5}
A solution $w_\varepsilon$ to the time-dependent dynamic programming principle (\ref{sTDPP}) exists and is unique. In addition, it satisfies
\begin{equation}
 |w_\varepsilon(t, \cdot) - \tilde{\varphi}|_\infty < C(T-t +1),
\end{equation}
with a constant $C$ that is independent of $\varepsilon$. Moreover,
\begin{enumerate}
\item $w_\varepsilon(t,x)$ grows at most linearly as $|x| \to \infty$ (with a bound that is uniform as $\varepsilon \to 0$)
\item The function $w_\varepsilon$  is monotone nondecreasing in each variable $x_j$
\item The function $w_\varepsilon$ satisfies translation invariance, i.e.
$$w_\varepsilon(t,x_1+c,x_2+c,...,x_n+c) = w_\varepsilon(t,x_1,x_2,...,x_n)+c.$$
\end{enumerate}
\end{theorem}
\begin{proof}
Existence and uniqueness follow directly from the dynamic programming principle: solutions are built one time step at a time: at levels $T, T-\varepsilon^2, T-2\varepsilon^2, ...$.  The proof of the estimate is by induction on the number of time steps.
For $t=T$,  $w_\varepsilon(t, \cdot) =\varphi$ by definition and the bound is an immediate consequence of our choice of  $\tilde{\varphi}$ (a smoothed out version of $\varphi$, see \ref{Prop2}). For the inductive step, suppose the bound holds at $t=T-k\varepsilon^2$, i.e.
$$ |w_\varepsilon(t, x) - \tilde{\varphi}(x)|_\infty < C(T-t +1) .$$
Then, let us consider what happens at $t-\varepsilon^2$. The argument used to prove (\ref{Emin}) shows that for the optimal choices of strategy by the market and the player, the following holds:
 $$E =| \mathbf{E}[\tilde{\varphi}(x+\varepsilon\Delta x)-\tilde{\varphi}(x)]  | \leq C\varepsilon^2.$$
We use this in the second line of the estimate:
\begin{eqnarray*}
|w_\varepsilon(t-\varepsilon^2,x)-\tilde{\varphi}| &=& |\min_{player}\max_{market}\mathbf{E}[w_\varepsilon(t, x+\varepsilon\Delta x)]-\tilde{\varphi}(x)|\\
&=& |\min_{player}\max_{market}\mathbf{E}[w_\varepsilon(t, x+\varepsilon\Delta x)-\tilde{\varphi}(x+\varepsilon\Delta x)+\tilde{\varphi}(x+\Delta x)-\tilde{\varphi}(x)]|\\
&\leq& |C(T-t +1)+\min_{player}\max_{market}(\tilde{\varphi}(x+\varepsilon\Delta x)-\tilde{\varphi}(x))|  \\
&=&  C(T-t +1)+C\varepsilon^2\\
&=&  C(T-(t-\varepsilon^2)+1).
\end{eqnarray*}
This concludes the inductive step.
\end{proof}
The symmetry, monotonicity, and translation invariance properties are easily established inductively, using arguments parallel to the one used for Lemma \ref{sym1}.

\section{Review of Known Results about Viscosity Solutions of our PDEs\label{chap:four}}\label{EU}

In section 3 we showed that the discrete solutions to the finite horizon and geometric stopping problems have at most linear growth as $|x| \rightarrow \infty$. We will prove in section \ref{CV} that the solutions converge as $\varepsilon \to 0$ to the viscosity solution of the appropriate PDE. Since the discrete solutions have linear growth as $|x| \to \infty$ (with a bound that is independent of $\varepsilon$), we only need to concern ourselves with at most linear growth solutions to the PDEs.

The existence and uniqueness of viscosity solutions of our PDE's (with at most linear growth at $\infty$) are well known. This short section provides the relevant definitions and results.

\subsection{The Time Dependent Case}
The following definitions are standard.
\begin{definition}
A real-valued, lower-semicontinuous function $w(t,x)$ defined for $x \in \RR^n$ and $t \leq T$ is a viscosity supersolution of the final-value problem \eqref{parE} if for any $(t_0, x_0)$ with $t_0 < T$ and any smooth $\psi(t,x)$ such that $w-\psi$ has a local minimum at $(t_0, x_0)$ we have
$$\psi_t(t_0,x_0) + \frac{1}{2}\max_{v\in\{0,1\}^n}\langle D^2\psi(t_0,x_0)\cdot v, v\rangle \leq 0,$$
and $w \geq \varphi$ at the final time $t=T$.
\end{definition}
\begin{definition}
A real-valued, upper-semicontinuous function $w(t,x)$ defined for $x \in \RR^n$ and $t \leq T$ is a viscosity subsolution of the final-value problem \eqref{parE} if for any $(t_0, x_0)$ with $t_0 < T$ and any smooth $\psi(t,x)$ such that $w-\psi$ has a local maximum at $(t_0, x_0)$ we have
$$\psi_t(t_0,x_0) + \frac{1}{2}\max_{v\in\{0,1\}^n}\langle D^2\psi(t_0,x_0)\cdot v, v\rangle \geq 0,$$
and $w \leq \varphi$ at the final time $t=T$.
\end{definition}
\begin{definition}
A viscosity solution of the final-value problem \eqref{parE} is a continuous function $w$ that is both a subsolution and a supersolution.
\end{definition}

\begin{theorem}\label{comp} The final-value problem \eqref{parE} - informally written as
$$w_t + \frac{1}{2}\max_{v\in\{0,1\}^n}\langle D^2w\cdot v, v\rangle = 0,$$
subject to $w(T,x) = \varphi(x)$ -  has a unique viscosity solution $w$ that grows at most linearly and $w$ is uniformly continuous.
Moreover, if $w_1$ is a subsolution, and $w_2$ is a supersolution, then necessarily $w_1 \leq w_2$.

\end{theorem}
\begin{proof}
The statement is a special case of theorem 2.1 in \cite{GGIS}, applied  backwards in time.
\end{proof}

\subsection{The Stationary Case}

Now we focus on viscosity solutions for the stationary equation. As before, the following definitions are well-known.
\begin{definition}
A real-valued, lower-semicontinuous function $u(x)$ defined for $x \in \RR^n$ is a viscosity supersolution of the stationary problem \eqref{Equa} if for any $x_0 \in \RR^n$ and any smooth $\psi(x)$ such that $u-\psi$ has a local minimum at $ x_0$ we have
$$\psi(x_0) -\varphi(x_0) -\frac{1}{2}\max_{v\in\{0,1\}^n}\langle D^2\psi(x_0)\cdot v, v\rangle \geq 0.$$
\end{definition}
\begin{definition}
A real-valued, upper-semicontinuous function $u(x)$ defined for $x \in \RR^n$ is a viscosity subsolution of the stationary problem \eqref{Equa} if for any $x_0 \in \RR^n$ and any smooth $\psi(x)$ such that $u-\psi$ has a local maximum at $x_0$ we have
$$\psi(x_0) -\varphi(x_0) -\frac{1}{2}\max_{v\in\{0,1\}^n}\langle D^2\psi(x_0)\cdot v, v\rangle \leq 0.$$
\end{definition}
\begin{definition}
A viscosity solution of \eqref{Equa} is a continuous function $u$ that is both a subsolution and a supersolution.
\end{definition}

\begin{theorem} \label{comp-elliptic} The stationary equation \eqref{Equa}, informally written as
$$u -\varphi -\frac{1}{2}\max_{v\in\{0,1\}^n}\langle D^2u\cdot v, v\rangle= 0,$$
has a unique viscosity solution $u$ that is uniformly continuous and grows at most linearly at infinity.
\end{theorem}
\begin{proof}
 We check that the conditions of Theorem 5.1 in \cite{CIL} hold: $\varphi(x) \in UC(\mathbb{R}^n)$ is of at most linear growth. Moreover,
$$\mathcal{L}(u)=-\frac{1}{2}\max_{v\in\{0,1\}^n}\langle D^2u\cdot v, v\rangle$$
 is degenerate elliptic by the Lemma \ref{dege}. This establishes the conditions of Theorem 5.1; we now conclude from \cite{CIL} that the elliptic equation has a unique viscosity solution that grows at most linearly as $|x| \to \infty$.
\end{proof}

%%%%%%%%%%%%%%%%%%%%%%%%%%%%%%%%%%%%%%%%%%%%%%%%%%%%%%%%%

\section{Convergence to the Viscosity Solution\label{chap:five}}\label{CV}

In this section we show that the solutions of our discrete problems converge to the viscosity solution of our PDEs as $\varepsilon \to 0$.
In order to do so, we follow the setup of Barles and Souganidis \cite{BS}. The essence of the Barles-Souganidis convergence result is that if an approximation scheme is monotone, stable, and consistent, then solutions converge as  $\varepsilon \to 0$ to the viscosity solution of the associated PDE. This section provides the argument in a self-contained form as it applies to our setting.
Following standard notation, in the geometric stopping case we write
$$\mathcal{F}_\varepsilon[u(x)] := F(\varepsilon, x, u_0, u),$$
%has Our discussion of the geometric stopping problem uses the scheme $$\mathcal{F}_\varepsilon[u](x)=F_\varepsilon^x(u_0, u_0-u_j).$$
The induction defining the finite-horizon problem solution $w_\varepsilon$ can also be viewed as solving a `scheme' and this viewpoint will be useful for analyzing the limit as $\varepsilon \to 0$. We define the finite horizon approximation scheme as:
\begin{eqnarray}
\hspace{-1cm}  \tilde{\mathcal{F}}_\varepsilon[w(t,x)]&:=& \frac{w(t,x)-\min_{player}\max_{market}\mathbf{E}[w(t+\varepsilon^2, x+\varepsilon\Delta x)]}{\varepsilon^2}.
\end{eqnarray}
Following standard notation here too, we write this as
$$ \tilde{\mathcal{F}}_\varepsilon[w(x)] := F(\varepsilon,t, x, w_0, w).$$

%An alternative notation is to write the scheme as a functional of its small parameter $\varepsilon$, a point $(t, x)$ in time-space, the value $w_0$ of a function $w$ at the point $(t,x)$, and the function $w$ itself:
% \begin{equation}
%\tilde{\mathcal{F}}_\varepsilon[w(t,x)] := \tilde{F}(\varepsilon, t, x, w_0, w).
%\end{equation}
\subsection{Monotonicity}

\begin{definition}
A time-independent scheme $\mathcal{F}_\varepsilon$ is monotone if
$$F(\varepsilon, x, u_0, u) \leq F(\varepsilon, x, u_0, v)$$
 whenever $u\geq v$ for all $\varepsilon \geq 0$, $x \in \RR^n$, $u_0 \in \RR$, and $u, v \in UC(\mathbb{R}^n)$.\\
A time-dependent scheme $\tilde{\mathcal{F}}_\varepsilon$ is monotone if
$$\tilde{F}_\varepsilon(\varepsilon, t, x, w_0, w) \leq \tilde{F}_\varepsilon(\varepsilon, t, x, w_0, v)$$
 whenever $w\geq v$ for all $\varepsilon \geq 0$,
$t<T$, $x \in \RR^n$, $w_0 \in \RR$, and $w, v \in UC(\mathbb{R}^n)$.
\end{definition}
\begin{lemma}
Our schemes $\mathcal{F}_\varepsilon$ and $\tilde{\mathcal{F}}_\varepsilon$ are monotone.
\end{lemma}
\begin{proof}
Firstly, let us prove the statement for the time-dependent scheme:
\begin{eqnarray*}
 \tilde{F}(\varepsilon, t, x, w_0, w)&:=& \frac{w_0 -\min_{player}\max_{market}\mathbf{E}[w(t+\varepsilon^2, x+\varepsilon\Delta x)]}{\varepsilon^2} \\
&\leq&\frac{w_0 -\min_{player}\max_{market}\mathbf{E}[v(t+\varepsilon^2, x+\varepsilon\Delta x)]}{\varepsilon^2} \\
& = & \tilde{F}(\varepsilon, t, x, w_0, v)
\end{eqnarray*}
The inequality follows from applying an expected value to $ w(t+\varepsilon^2, x+\varepsilon\Delta x) \geq v(t+\varepsilon^2, x+\varepsilon\Delta x)$ and reversing signs.

Next, we prove the statement for the stationary scheme:
\begin{eqnarray*}
 F(\varepsilon, x, u_0, u)&:=& u_0 - \varphi(x) -\frac{1-\varepsilon^2}{\varepsilon^2}\min_{player}\max_{market}\mathbf{E}[u(x+\varepsilon\Delta x)-u_0] \\
&\leq&u_0 - \varphi(x) -\frac{1-\varepsilon^2}{\varepsilon^2}\min_{market}\max_{market}\mathbf{E}[v(x+\varepsilon\Delta x)-u_0] \\
& = & F(\varepsilon, x, u_0, v)
\end{eqnarray*}
 The inequality follows from applying the expected value to $ u(x+\varepsilon\Delta x) -u_0\geq v(x+\varepsilon\Delta x) -u_0$.
\end{proof}

\subsection{Main Result}

As already mentioned, \cite{BS} shows that if a numerical scheme is stable, monotone, and consistent then its solutions converge to those of the associated PDE. In this paper stability is provided by Theorems \ref{Euler} and \ref{Th5}, which proves uniform bounds on $u_\varepsilon$ and $w_\varepsilon$ (independent of $\varepsilon$). The heuristic argument in Section 2 provides the essence of the argument for consistency  (taking into account that $u_\varepsilon$ and $w_\varepsilon$ are increasing in each $x_i$ and satisfy the "translation property").  A more rigorous proof of consistency will be part of the proof of the following convergence theorem.

\begin{theorem}\label{Main}
The unique solutions $u_\varepsilon$ and $w_\varepsilon$ of $\mathcal{F}_\varepsilon$ and $\tilde{\mathcal{F}}_\varepsilon$ converge to the unique solutions of
 (\ref{Equa}) and  (\ref{parE}), respectively.
\end{theorem}
\begin{proof}
The first part of the proof follows \cite{BS} and \cite{CIL}. We do the proof in the time dependent case (the stationary case is identical). We define $\overline{w}$, $\underline{w}$ by
\begin{equation}
\overline{w}(t,x)= \limsup_{\substack{\varepsilon \to 0,~ \tau \to t\\ y \to x}}w_\varepsilon(\tau, y)
\end{equation}
and
\begin{equation}
\underline{w}(t,x)= \liminf_{\substack{\varepsilon \to 0,~ \tau \to t\\ y \to x}}w_\varepsilon(\tau, y).
\end{equation}
The functions $\overline{w}$ and $\underline{w}$ have the translation property and are monotone in each variable because the sequences $w_\varepsilon$ have those properties.
We prove that  $\overline{w}(t,x)$ is a sub-solution  (the proof that $\underline{w}(t,x)$ is a supersolution is completely parallel).

Consider  $\xi \in C^\infty$, which touches  $\overline{w}(t,x)$ at $(t_0,x_0)$ - a local maximum of $\overline{w}(t,x)-\xi(t,x)$; we also assume $t_0 < T$ (the other case $t_0=T$ is presented towards the end of this proof).  To make notation simpler we can modify $\xi$,  (without loss of generality) so that (i) $\overline{w}-\xi$ has a  maximum at $(t_0,x_0)$ and (ii) $\overline{w}(t_0,x_0)-\xi(t_0,x_0)=0$.

We change coordinates so that $\tilde{x}=\pi(x)$ is the projection of $x$ orthogonal to $(1,1,...,1)$ whereas $z:=\frac{1}{n}(x_1+...+x_n)$, is the projection of $x$ onto $(1,1,...,1)$. Since $\overline{w}$  has the translation property, there is a unique function $\tilde{w}(t,\tilde{x})$ defined for $\tilde{x} \in \{ x_1+...+x_n\}=0$ such that $\overline{w}(t,x)=\tilde{w}(t,\pi(x)) +z$.

We fix a $\delta$ and employ Theorem 3.2 from \cite{CIL}. We obtain a sequence of functions $\tilde{\psi}_j(t,\tilde{x})$ with the following properties:
\begin{enumerate}
\item $\tilde{\psi}_j(t,\tilde{x})$ touches $\tilde{w}$ at $(t'_j, x'_j)$ near $(t_0, \pi(x_0))$, so $\tilde{w}-\tilde{\psi}_j$ has a strict local maximum at  $(t'_j, x'_j)$ and without loss of generality $\tilde{w}(t'_j, x'_j)-\tilde{\psi}_j(t'_j, x'_j)=0$.
\item The first derivatives of $\tilde{\psi}_j$ at $(t'_j, x'_j)$ converge (as $j \to \infty$) to the first derivatives of $\xi$ at $(t_0, \pi(x_0),0)$.
\item The second derivative matrix $X_j$ of $\tilde{\psi}_j(t,\tilde{x})$  at $(t'_j, x'_j)$ (with respect to spatial variables $\tilde{x}$) converges to the matrix $X$, which satisfies
\begin{equation}\label{mpsi}
\tilde{X}=
\begin{pmatrix}
X & 0\\
0 & 0
\end{pmatrix}
\leq A +C\delta.
\end{equation}
where $A=D^2\xi(t_0, x_0)$ is the Hessian of $\xi$ at $(t_0,x_0)$ in the variables $z, \tilde{x}$ and $C$ is a constant depending on $\xi$ only.
\end{enumerate}
We extend $\tilde{\psi}_j$ to $\psi_j(t,x) = \tilde{\psi}_j(t,\pi(x))+ z$, and observe that $\psi_j(t,x)$ has the same second derivatives as $\tilde{\psi}_j$ with respect to $t, \tilde{x}$, as well as $\partial_{zz}\psi_j=0$. We observe that
$$(\partial_1 + ... + \partial_n)\psi_j = 1,$$
by construction, regardless of location. Therefore, we can differentiate the above expression, obtaining, for every $k$,
\begin{equation}\label{der}
\partial_k(\partial_1 + ... + \partial_n)\psi_j = 0.
\end{equation}
We will use this relation in the argument below.

We argue similarly to \cite{BS}. Consider
$$\overline{w} -\psi_j = \tilde{w}+z - (\tilde{\psi}_j+z)= \tilde{w}-\tilde{\psi}_j,$$
which implies that $\psi_j$ touches $\overline{w}$ whenever $\tilde{\psi}_j$ touches $\tilde{w}$. In particular, $\psi_j$ touches $\overline{w}$  at $(t'_j, x'_j,z)$ for any $z$.
Since $\tilde{w}-\tilde{\psi}_j$ has a local max at $(t'_j, x'_j)$ there exists a ball $B(t'_j, x'_j,r)$ with radius $r$, so that $\tilde{w}-\tilde{\psi}_j<0$ on the ball. Moreover, because we want the local maximum to be a global maximum, we can change $\tilde{\psi}_j(t,x)$ so that
 $$\tilde{\psi}_j(t,\tilde{x}) \geq (\tilde{\varphi}(\tilde{x},0)+C(T-t+1)) \geq \sup_{\varepsilon}w_\varepsilon$$
outside the ball  $B(t'_j,x'_j, r).$ The second inequality is a consequence of theorem \ref{Th5}. The function $
\tilde{\varphi} = \varphi*\eta$ is the smooth version of $\varphi$, introduced in subsection \ref{smooth}.
After the adjustment of $\tilde{\psi}_j$ we obtain that $(t'_j, x'_j)$ is a global max of $\tilde{w}-\tilde{\psi}_j$.

Since $\overline{w}(t,x)= \limsup_{\substack{\varepsilon \to 0,~ \tau \to t\\ y \to x}}w_\varepsilon(\tau, y)$ and $w_\varepsilon$ has the translation property (i.e. $w_\varepsilon(t,x)= \tilde{w}_\varepsilon(t,\pi(x)) +z$), we can obtain
sequences $\varepsilon_n$ and $(\tau_n, y_n)$  such that $\pi(y_n)=0$ and
\begin{enumerate}
\item  $\tilde{w}_{\varepsilon_n}- \tilde{\psi}_j$  achieves its global max at $(\tau_n, y_n)$
\item $(\tau_n, y_n) \to (t'_j, x'_j)$
\item $w_{\varepsilon_n}(\tau_n,y_n) \to \tilde{w}(t'_j, x'_j)$.
\end{enumerate}
Denote $\theta_n = \tilde{w}_{\varepsilon_n}(\tau_n,y_n)-\tilde{\psi}_j(\tau_n,y_n).$
Since we have global maxima, we obtain
$$w_{\varepsilon_n}(\tau,y)-\psi_j(\tau,y) \leq \theta_n, $$
or equivalently
$$w_{\varepsilon_n}(\tau,y) \leq \psi_j(\tau,y) +\theta_n. $$

We are prepared to use the properties of the scheme:
\begin{eqnarray*}
 0 &=&  \tilde{F}(\varepsilon_n, \tau_n, y_n, w_{\varepsilon_n}(\tau_n,y_n), w_{\varepsilon_n})\\
 &=&  \tilde{F}(\varepsilon_n, \tau_n, y_n, \psi_j(\tau_n,y_n)+\theta_n, w_{\varepsilon_n})\\
&\geq&  \tilde{F}(\varepsilon_n, \tau_n, y_n, \psi_j(\tau_n,y_n)+\theta_n, \psi_j +\theta_n).
\end{eqnarray*}
The equalities follows from $w_{\varepsilon_n}$ being a solution to the scheme, while the inequality follows from monotonicity with respect to the larger function  $  \psi_j(\tau,y) +\theta_n$.
Now we take limits in order to apply consistency of the scheme:
\begin{eqnarray*}
 0 &\geq& \lim_j\varlimsup_n \tilde{F}(\varepsilon_n, \tau_n, y_n, \psi_j(\tau_n,y_n)+\theta_n, \psi_j +\theta_n ) \\
&=& \lim_j\varlimsup_n \frac{\psi_j(\tau_n,y_n)-\min_{player}\max_{market}\mathbf{E}[\psi_j(\tau_n+\varepsilon_n^2, y_n+\varepsilon_n\Delta y_n)]}{\varepsilon_n^2} \\
&= & \lim_j\varlimsup_n -\frac{1}{\varepsilon_n^2}\min_{player} \max_{market} \mathbf{E}[\varepsilon_n\langle\nabla \psi_j, \Delta y \rangle+ \varepsilon_n^2(\psi_{jt} +\frac{1}{2} \langle D^2\psi_j \cdot \Delta y, \Delta y\rangle) + o(\varepsilon_n^2)]
\end{eqnarray*}

We begin with two observations. First, $o(\varepsilon_n^2)$ divided by the $\varepsilon_n^2$ denominator is insignificant as it vanishes in the limit; thus the  $o(\varepsilon_n^2)$ can (and will) be ignored in what follows. Our second observation is that the $\varepsilon_n^2$ term
$$ \varepsilon_n^2(\psi_{jt} +\frac{1}{2} \langle D^2\psi_j \cdot \Delta y, \Delta y\rangle)$$
can be simplified using translation invariance; in fact it can be rewritten into its PDE form in an entirely parallel fashion to the one used in the heuristic derivation found in subsection 2.3. In particular, its value depends only on the market's choices (not the player's choices).

Observe that the $\min$ over all the player's choices is less than or equal to the expression with a particular choice of the player. Thus
$$\min_{player} \max_{market}\frac{ \mathbf{E}[\varepsilon_n\langle\nabla \psi_j, \Delta y \rangle+ \varepsilon_n^2(\psi_{jt} +\frac{1}{2} \langle D^2\psi_j \cdot \Delta y, \Delta y\rangle) ]}{\varepsilon^2_n}$$
is less than or equal to the value of
\begin{eqnarray} \label{expect}
\max_{market}\frac{ \mathbf{E}[\varepsilon_n\langle\nabla \psi_j, \Delta y \rangle+ \varepsilon_n^2(\psi_{jt} +\frac{1}{2} \langle D^2\psi_j \cdot \Delta y, \Delta y\rangle) ]}{\varepsilon^2_n}
\end{eqnarray}
when the player chooses the particular strategy
$$\alpha_i =\partial_i\psi_j,~~~i=1,...,n.$$
Note that we use that
$\sum_{k=1}^n \partial_k \psi =1 $ and $\partial_i \psi_j \geq0$ for $k=1,...,n$.
The equality comes from $\psi_j$ having the translation property; the inequalities $\partial_i \psi_j \geq0$ follow by a standard argument from the facts that $w_{\varepsilon_n}$ is nondecreasing in $x_i$, and
that $w_{\varepsilon_n}  - \psi_j$ has a local maximum at the point around which we perform the Taylor expansion.

The expression \ref{expect} seems to have a term proportional to $\varepsilon_n^{-1}$, i.e.
 $$ \varepsilon_n^{-2}\mathbf{E}[\varepsilon_n\langle\nabla \psi_j, \Delta y \rangle].$$
However, for the particular choice of values for $\alpha$ this term vanishes as shown in subsection \ref{PDEGS}:
$$\mathbf{E}_{\vec{\alpha}, p}[\langle\nabla u , \Delta x \rangle] =  \sum_{i=1}^{n}[  \partial_i u - \alpha_i ]\mathbf{E}_{p}v_i.$$
 Thus (\ref{expect})  is actually equal to :
$$\max_{market}(\psi_{jt} +\frac{1}{2} \langle D^2\psi_j \cdot \Delta y, \Delta y\rangle),$$
which using the arguments in subection \ref{el} equals
$$ \psi_{jt}+\frac{1}{2}\max_{v\in\{0,1\}^n} \langle D^2\psi_j\cdot v, v\rangle.$$
We conclude that:

\begin{eqnarray*}
&& \lim_j\varlimsup_n -\frac{1}{\varepsilon_n^2}\min_{player} \max_{market} \mathbf{E}[\varepsilon_n\langle\nabla \psi_j, \Delta y \rangle+ \varepsilon_n^2(\psi_{jt} +\frac{1}{2} \langle D^2\psi_j \cdot \Delta y, \Delta y\rangle) + o(\varepsilon_n^2)]\\
&\geq& \lim_j( -\psi_{jt} (t'_j,x'_j,0) -\frac{1}{2}\max_{v\in\{0,1\}^n}\langle D^2\psi_j(t'_j,x'_j)\cdot v, v\rangle)\\
&=&-\psi_t (t_0,x_0) -\frac{1}{2}\max_{v\in\{0,1\}^n}\langle \tilde{X}\cdot v, v\rangle\\
& \geq&-\xi_t ({t_0},{x_0}) -\frac{1}{2}\max_{v\in\{0,1\}^n}\langle D^2\xi(t_0,x_0)\cdot v, v\rangle -K\delta.
\end{eqnarray*}

The equalities above essentially follow the heuristic argument in section 2:  applying the definition, canceling terms, and Taylor expansion.
The last inequality follows, because $\lim_j\psi(t_0, \tilde{x_0},0)$ and $\xi(t_0,x_0)$ have matching time derivatives by construction, and because the matrix comparison in \ref{mpsi} holds.
In the expression above we may chose $\delta$ as small as we like; sending it to 0 completes the proof that $\overline{w}$ is a supersolution for $t<T$.

Finally, let us consider the final time $t_0=T$ for the time-dependent case. We need to show that  $\overline{w}(T,x)\leq \varphi(T,x)$. In fact, we will prove that $\overline{w}(T,x)= \varphi(T,x)$. Because of the translation property, we can examine
points $x_0$ such that $\sum_j x_{0,j}=0$, and a barrier function $\tilde{\psi}$, such that
$$\tilde{\psi} (t,\tilde{x}) = \frac{|\tilde{x}-x_0|^2}{\delta}+\frac{T-t}{\mu}$$
for  $\sum_j \tilde{x}_j=0$. Just as before, we extend $\tilde{\psi}$ and $\tilde{w}$ so that
 $$\psi(t,x) = \tilde{\psi}(t,\pi(x))+ z$$
and
$$ \overline{w}(t,x)=\tilde{w}(t, \pi(x))+z.$$
Since
$$\overline{w} -\psi = \tilde{w}+z - (\tilde{\psi}+z)= \tilde{w}-\tilde{\psi},$$
we can focus on maximizing $ \tilde{w}-\tilde{\psi}$ (and not $\overline{w} -\psi$). We consider the half-space $(t\leq T, \sum_j x_j=0)$ and let
$(\tau_{\delta,\mu}, x_{\delta,\mu})$ be the point where maximum of $ \tilde{w}-\tilde{\psi}$ attains its max. We see that
$$(\tau_{\delta,\mu}, x_{\delta,\mu}) \to  (T,x_0)  ~~~\mbox{as}~~ ~\delta,\mu \to 0. $$
Moreover,
\begin{equation}\label{ineq2}
\tilde{w}(\tau_{\delta,\mu}, x_{\delta,\mu})\geq \tilde{w}(\tau_{\delta,\mu}, x_{\delta,\mu})-\tilde{\psi}(\tau_{\delta,\mu}, x_{\delta,\mu})\geq \tilde{w}(T,x_0)
.\end{equation}
Because of the above and $\tilde{w} =\limsup \tilde{w}_\varepsilon$, we see that
\begin{equation}\label{conv}
\tilde{w}(\tau_{\delta,\mu}, x_{\delta,\mu}) \to \tilde{w}(T, x_0)~~~\mbox{as}~~ ~\delta,\mu \to 0.
\end{equation}
Consider the maximum point $\tau_{\delta,\mu}, x_{\delta,\mu}$. If $\tau_{\delta,\mu}<T$, then we repeat the argument presented above for the interior case to get
\begin{equation}\label{tpde}
0=-\psi_t (T,x_0) -\frac{1}{2}\max_{v\in\{0,1\}^n}\langle D^2\psi(T,x_0)\cdot v, v\rangle.
\end{equation}
We restrict our attention to choices of $\delta>0, \mu>0$ so that $\delta - n\mu>0$. Then,
\begin{equation}\label{no}
0=-\psi_t (\tilde{t_0},\tilde{x_0}) -\frac{1}{2}\max_{v\in\{0,1\}^n}\langle D^2\psi(\tilde{t_0},\tilde{x_0})\cdot v, v\rangle\geq \frac{1}{\mu}-\frac{n}{\delta}=\frac{\delta-n\mu}{\mu \delta},
\end{equation}
which is a contradiction. Therefore, if $(\delta, \mu) \to 0$ with $0< \delta -\mu n$, then $\tau_{\delta,\mu} = T$ when $\delta$ and $\mu$ are sufficiently small.
We have leftover to prove that $\overline{w}(T,x_0)=\varphi(x_0)$; in order to do that, by \ref{conv} it is enough to show that $\tilde{w}(T,x_{\delta, \mu})=\varphi(x_{\delta,\mu})$, when $\delta,\mu$ - sufficiently small.
The proof is parallel to the one of the interior case. We use that
\begin{equation}
\overline{w}(T,x)= \limsup_{\substack{\varepsilon \to 0,~ \tau \to T\\ y \to x}}w_\varepsilon(\tau, y)
\end{equation}
and that $w_\varepsilon$ has the translation property to obtain sequences $\varepsilon_n$ and $(\tau_n, y_n)$, for which $\pi(y_n)=0$ and
\begin{itemize}
\item  $\tilde{w}_{\varepsilon_n}$ is maximized on $t \leq T$ at $(\tau_{n}, y_{n})$
\item  $(\tau_{n}, y_{n}) \to  (T,x_{\delta,\mu})  $
\item  $\tilde{w}_{\varepsilon_n} \to \tilde{w}(T,x_{\delta,\mu}).$
\end{itemize}
If $\tau_n<T$ for infinitely many $\tau_n$, we obtain equation (\ref{no}), a contradiction.
Hence, for all large $n$ we obtain $\tau_n=T$, which implies $\tilde{w}_{\varepsilon_n}=\varphi(y_n)$. Combining with the fact that $\varphi$ is continuous, we deduce that $\tilde{w}(T,  x_{\delta,\mu})=\varphi( x_{\delta,\mu})$. This concludes the proof that $\overline{w}$ is a subsolution.

As already mentioned, the proof that $\underline{w}$ is a supersolution is parallel. The main difference is working with the optimal choice for the market instead of the player.

We would like to show that  $\overline{w}=\underline{w}$ is the unique viscosity solution to the PDE (\ref{parE}).
One inequality comes from comparison principle: since $\overline{w}(t,x)$ is a upper semicontinuous sub-solution, as we just proved, and $\underline{w}(t,x)$ is a lower semicontunous super-solution, then by comparison principle (Theorem \ref{comp}) we obtain the desired inequality $\overline{w}(t,x) \leq \underline{w}(t,x)$.
The other inequality follows by the definition of $\limsup$ and $\liminf$. Therefore  $\overline{w}(t,x) = \underline{w}(t,x)=w$, which is what we wanted to show.
\end{proof}

\subsection{Consequences of the main result}
We proved that $\lim w_\varepsilon=w$ and $\lim u_\varepsilon=u$.  As a result a lot of the properties of the solutions to the discrete problem are inherited.

\begin{lemma}
The solution $w$ of the time-dependent problem \eqref{parE}  is symmetric, monotone, and translation invariant, ie if  $\tilde{x}_1\geq x_1$ and $c$ - any constant, then
\begin{eqnarray}
w(t,x_1,x_2,...,x_n) = w(t,x_2,x_1,...,x_n),\\
w(t,\tilde{x}_1,x_2,...,x_n) \geq w(t,x_1,x_2,...,x_n),\\
w(t,x_1+c,x_2+c,...,x_n+c) = w(t,x_1,x_2,...,x_n)+c.
\end{eqnarray}
\end{lemma}
\begin{proof}
We observe that $w_\varepsilon \to w$ as $\varepsilon \to 0$ by Theorem \ref{Main}, so we pass the equality through the limit, obtaining in the end the desired identities.
\end{proof}

\begin{lemma}\label{GEc}
The solution $u$ of the elliptic PDE (\ref{Equa})  is symmetric, monotone, and translation invariant, ie if  $\tilde{x}_1\geq x_1$ and $c$ - any constant, then
\begin{eqnarray}
u(x_1,x_2,...,x_n) = u(x_2,x_1,...,x_n).\\
u(\tilde{x}_1,x_2,...,x_n) \geq u(x_1,x_2,...,x_n)\\
u(x_1+c,x_2+c,...,x_n+c) = u(x_1,x_2,...,x_n)+c.
\end{eqnarray}
\end{lemma}
\begin{proof}
We observe that $u_\varepsilon \to u$ as $\varepsilon \to 0$ by theorem \ref{Main}, so we pass the equality through the limit, obtaining in the end the desired identities.
\end{proof}

%%%%%%%%%%%%%%%%%%%%%%%%%%%%%%%%%%%%%%%%%%%%%%%%%%%%%%%%%%%%%%

\section{Exact Solution\label{chap:six}}\label{ES}

It is natural to ask how the PDE might be used. We offer two simple applications in this section: an exact solution of the geometric stopping case for $n=3$ experts and a demonstration that the associated argument does not generalise straightforwardly to $n=4$ experts. (There is now an explicit solution for the geometric
stopping case with $n=4$ experts \cite{BEZ}. Its derivation makes use of our PDE.)

\subsection{The Geometric Stopping Case with $n=3$}\label{Explicit}

The following result is a continuous analogue of one in \cite{GPS}.

\begin{theorem}
The solution of our PDE \eqref{Equa} in the geometric stopping case for $n=3$ experts and $\varphi =\max\{x_1,x_2, x_3\}$ is symmetric with respect to $x_1, x_2, x_3$, and in the quadrant where $x_1 \geq x_2 \geq x_3$, its formula is
\begin{equation} \label{ExS}
u(x) = x_1 + \frac{1}{2\sqrt{2}}e^{\sqrt{2}(x_2-x_1)} + \frac{1}{6\sqrt{2}}e^{\sqrt{2}(2x_3-x_2-x_1)}.
\end{equation}
\end{theorem}
\begin{proof}
Since by Theorem \ref{comp-elliptic} the PDE \eqref{Equa} has a unique at most linear growth solution, all we need to do is verify that $u(x)$, which has linear growth, is a $C^2$ solution.

First, let us establish that the expression $u(x)$ is a solution within a quadrant. One can differentiate the formula to find the first derivatives:
\begin{eqnarray*}
\partial_1 u &=& 1 -\frac{1}{2}e^{\sqrt{2}(x_2-x_1)} - \frac{1}{6}e^{\sqrt{2}(2x_3-x_2-x_1)},\\
\partial_2 u &=&  \frac{1}{2}e^{\sqrt{2}(x_2-x_1)} - \frac{1}{6}e^{\sqrt{2}(2x_3-x_2-x_1)},\\
\partial_3 u &=&  \frac{2}{6}e^{\sqrt{2}(2x_3-x_2-x_1)}.
\end{eqnarray*}
We see that indeed $\partial_1 u+\partial_2 u + \partial_3 u = 1$ as expected.
The interesting $v$ are $(1,0,0), (0,1,0)$, and $(0,0,1)$, i.e.
$$\mathcal{L}(u) = - \frac{1}{2}\max\{\partial_1^2u, \partial_2^2u,\partial_3^2u, 0\}$$
and we find second derivatives
\begin{eqnarray*}
\partial_{11} u &=&  \frac{1}{\sqrt{2}}e^{\sqrt{2}(x_2-x_1)} + \frac{1}{3\sqrt{2}}e^{\sqrt{2}(2x_3-x_2-x_1)}\\
\partial_{22} u &=&  \frac{1}{\sqrt{2}}e^{\sqrt{2}(x_2-x_1)} + \frac{1}{3\sqrt{2}}e^{\sqrt{2}(2x_3-x_2-x_1)}\\
\partial_{33} u &=&  \frac{2 \sqrt{2}}{3}e^{\sqrt{2}(2x_3-x_2-x_1)}.
\end{eqnarray*}
Hence, in this quadrant $$\partial_{11}u=\partial_{22}u \geq \partial_{33}u \geq 0.$$
Plugging into the PDE, we establish that
$$u(x)=x_1 + \frac{1}{2\sqrt{2}}e^{\sqrt{2}(x_2-x_1)} + \frac{1}{6\sqrt{2}}e^{\sqrt{2}(2x_3-x_2-x_1)}= \max\{x_1,x_2, x_3\} - \mathcal{L}(u),$$
hence $u(x)$ is a solution to the PDE in this quadrant, and by symmetry in all quadrants.

All we have left to show is that the expression $u(x)$ stays $C^2$ across the surfaces bounding the quadrants, and at the origin.
Observe that the expression is $C^2$, with bounded third derivatives away from the surfaces ($x_1=x_2>x_3$ and $x_1>x_2=x_3$).
The expression is symmetric across the surfaces, and even, as
$$\partial_1 u|_{x_1=x_2}=1 -\frac{1}{2}e^{0\sqrt{2}} - \frac{1}{6}e^{\sqrt{2}(2x_3-x_1-x_1)}= \frac{1}{2}- \frac{1}{6}e^{\sqrt{2}(2x_3-x_1-x_1)} =\partial_2 u|_{x_1=x_2} $$
and
$$\partial_2 u|_{x_2=x_3}=\frac{1}{2}e^{\sqrt{2}(x_2-x_1)} - \frac{1}{6}e^{\sqrt{2}(x_3-x_1)}=\frac{2}{6}e^{\sqrt{2}(x_3-x_1)}=\partial_3 u|_{x_2=x_3} .$$
Because the expression is even, it is $C^2$ across these surfaces. It remains to show that $u$ is $C^2$ at the origin. Let us consider the Taylor expansion of the function in the quadrant $x_1 \geq x_2 \geq x_3$.
It is $$u(x)= \frac{\sqrt{2}}{3}+\frac{x_1+x_2+x_3}{3}+ \frac{\sqrt{2}}{3} [x_1^2+x_2^2+x_3^2 -(x_1x_2 + x_1x_3+x_2x_3)] + ...,$$
which is a symmetric function up to second order, with bounded third derivatives. By symmetry, the second order part of the Taylor expansion is the same in other sectors. Thus the formula is $C^2$ at the origin. Thus the function is $C^2$ at the origin, as well as everywhere else.
We established that $u(x)$ is a $C^2$ solution of the PDE.
\end{proof}

Now that we have presented the solution to the $n=3$ geometric stopping problem, we analyze which strategy the solution corresponds to (see subsection \ref{OptStr}). On the quadrant $x_1 \geq x_2 \geq x_3$, the solution $u$ has:
\begin{equation*}
\frac{1}{2}\max_{v\in\{0,1\}^n} \langle D^2u\cdot v, v\rangle = \frac{1}{2}\partial_{11}u = \frac{1}{2}\partial_{22}u.
\end{equation*}
Since $\partial_{11}u=\langle D^2u\cdot v, v\rangle $ when $v=(1, 0, 0)$ and $\partial_{22}u=\langle D^2u\cdot v, v\rangle $ when $v=(0, 1,  0),$ we see (via the discussion in Section \ref{OptStr}) that the market has two optimal strategies:
\begin{itemize}
\item choose $(1, 0, 0)$ and $(0, 1, 1)$ with probability $1/2$ each, or
\item choose $(0, 1, 0)$ and $(1, 0, 1)$ with probability $1/2$ each.
\end{itemize}
% $$u = \frac{1}{2}\partial_{11}u + \varphi(x)$$ we have that $$\mbox{argmax} v = v_1 = (1,0,0).$$ Then, $$\tilde{v}_1 = (1,1,1) - v_1 = (0, 1, 1).$$
%
%In order to have a balanced strategy, we choose $v_1$ with probability a half, and $\tilde{v}_1$ with probability a half. Thus this strategy is the one corresponding to $$u(x) = \frac{1}{2}\max_{v\in\{0,1\}^n} \langle D^2u\cdot v, v\rangle + \varphi(x) = \frac{1}{2}\partial_{11}u + \varphi(x). $$
%
%Similarly, since also $ \frac{1}{2}\max_{v\in\{0,1\}^n} \langle D^2u\cdot v, v\rangle  = \frac{1}{2}\partial_{22}u,$ we obtain the other pure optimal strategy
%$\mbox{argmax} v = v_2 = (0,1,0).$ and $\tilde{v}_2 = (1,1,1) - v_2 = (1, 0, 1),$ each of which the market chooses with probability a half.
%Note that $v_3 = (0, 0, 1)$ is not optimal.

How could one find the explicit solution \eqref{ExS}? Well, suppose we know the optimal strategy $v^*$ on a region $\sigma$. Then,
\begin{equation*}
u(x) = \frac{1}{2} \langle D^2u\cdot v^*, v^*\rangle + \varphi(x)
\end{equation*}
is the corresponding PDE (or ODE). We know that the solutions of $u(x) =  \frac{1}{2}\partial_{kk}u+ \varphi(x)$ involve exponentials, so we expect a solution of the form
$$ u(x) = \varphi(x) + \sum c_k \exp (a_{k1} x_1 + ... + a_{kn}x_n).$$ The boundary condition of at most linear growth at infinity helps rule out the exponentials that grow at infinity, whereas the boundary conditions on the walls $x_1=x_2$, $x_2=x_3$ helps one determine the explicit solution formula.

\subsection{The Geometric Stopping Case with $n=4$ is different}

It is natural to ask whether the geometric stopping case with $n=4$ experts
(and $\varphi(x) = \max\{x_1, x_2, x_3, x_4\}$) can be solved explictly by
making an educated guess based on what we just did for $n=3$. We show in this section that
the answer is no. (In fact an exact solution for the geometric stopping case with $n=4$ experts
is now known; it was found by \cite{BEZ}, using our PDE characterization and arguments much
more involved than those in this section.)

Recall that for $n=3$, one of the market's two optimal strategies in the sector $x_1>x_2>x_3$ was to
advance the leading expert (i.e. take $v=(1,0,0)$) with probability $1/2$, and to advance everyone else
(i.e. take $v=(0,1,1)$) with probability $1/2$. With this in mind, we ask whether for $n=4$ it would be optimal
in the sector $x_1>x_2>x_3>x_4$ for the market to advance the leading expert with probability $1/2$ and
advance everyone else with probability $1/2$. If so, then in this sector the value function would satisfy
$$
u(x) =  \frac{1}{2}\partial_{11}u+ \max\{x_1, x_2, x_3, x_4\}
$$
It must also have linear growth at infinity, and at the sector's boundaries symmetry demands that
$\partial_1u=\partial_2u$ when $x_1=x_2$, $\partial_2u=\partial_3u$ when $x_2=x_3$,
and $\partial_3u=\partial_4u$ when $x_3=x_4.$ These conditions fully determine the
function; after some calculation, one obtains
\begin{equation}
u_4(x) = x_1 + \frac{1}{2\sqrt{2}}e^{\sqrt{2}(x_2-x_1)} + \frac{1}{6\sqrt{2}}e^{\sqrt{2}(2x_3-x_2-x_1)} + \frac{1}{12\sqrt{2}}e^{\sqrt{2}(3x_4 - x_3-x_2-x_1)}.
\end{equation}
We shall show that the proposed strategy is not optimal (and $u_4$ is not the value function in the
sector $x_1>x_2>x_3>x_4$) by showing that
$\partial_{11} u_4  \neq \max_{v\in\{0,1\}^n} \langle D^2u_4\cdot v, v\rangle$ in part
of this sector. It suffices to show that
$$
\partial_{11}u_4 < (\partial_{1}+ \partial_2)^2u_4
$$
in part of the sector. Explicit calculation gives
\begin{align*}
\partial_{11}u_4 =&  \frac{1}{\sqrt{2}}e^{\sqrt{2}(x_2-x_1)} + \frac{1}{3\sqrt{2}}e^{\sqrt{2}(2x_3-x_2-x_1)} + \frac{1}{6\sqrt{2}}e^{\sqrt{2}(3x_4 - x_3-x_2-x_1)},\\
\partial_{22}u_4 =&  \frac{1}{\sqrt{2}}e^{\sqrt{2}(x_2-x_1)} + \frac{1}{3\sqrt{2}}e^{\sqrt{2}(2x_3-x_2-x_1)} + \frac{1}{6\sqrt{2}}e^{\sqrt{2}(3x_4 - x_3-x_2-x_1)}, \ \mbox{and}\\
\partial_{12}u_4 =& - \frac{1}{\sqrt{2}}e^{\sqrt{2}(x_2-x_1)} + \frac{1}{3\sqrt{2}}e^{\sqrt{2}(2x_3-x_2-x_1)} + \frac{1}{6\sqrt{2}}e^{\sqrt{2}(3x_4 - x_3-x_2-x_1)}.
\end{align*}
Therefore
$$
(\partial_{1}+ \partial_2)^2u_4 = \partial_{11}u_4 + 2\partial_{12}u_4 + \partial_{22}u_4 =  \frac{4}{3\sqrt{2}}e^{\sqrt{2}(2x_3-x_2-x_1)} + \frac{4}{6\sqrt{2}}e^{\sqrt{2}(3x_4 - x_3-x_2-x_1)}.
$$
Evidently, when $x_1=x_2=x_3=x_4$, $\partial_{11} u_4 = \frac{3}{2\sqrt{2}}$,
which is strictly smaller than $(\partial_1 + \partial_2)^2 u_4 = \frac{2}{\sqrt{2}}$. So
(by continuity)
$$
\partial_{11}u_4 < (\partial_{1}+ \partial_2)^2u_4
$$
in part of the sector near $x_1=x_2=x_3=x_4$. Thus the proposed strategy is
not optimal, and $u_4$ is not the value function in this sector.

%%%%%%%%%%%%%%%%%%%%%%%%%%%%%%%%%%%%%%%%%%%%%%%%%%%%%%%%%%%%%%%%%%%
\newpage

\bibliography{MStrategies}
%bibliographysryle{plain}
\bibliographystyle{ieeetr}

\end{document}